\documentclass{article}
\usepackage[journal=RMI,lang=british]{ems-journal-rmi} 

\usepackage{enumerate}
\usepackage{pinlabel}

\DeclareMathOperator{\Length}{Length}

\newtheorem{thm}{Theorem}[section]
\numberwithin{thm}{section}
\newtheorem{defn}[thm]{Definition}

\newtheorem{cor}[thm]{Corollary}
\newtheorem{lem}[thm]{Lemma}

\newtheorem*{theorem*}{Theorem}
\newtheorem{theorem}{Theorem}

\newtheorem{corollary}[theorem]{Corollary}

\numberwithin{equation}{section}

\begin{document}

\title{Bounding Shortest Closed Geodesics\protect\\ with Diameter on compact\protect\\ 2-dimensional Orbifolds Homeomorphic to $S^2$}
\titlemark{Bounding Shortest Closed Geodesics}

%
\emsauthor{1}{
	\givenname{Jinxuan}
	\surname{Chen}
	\zblid{chen.jinxuan}
	\orcid{0009-0004-9867-1845}}{J.~Chen}

\Emsaffil{1}{
	\pretext{}
	\department{BICMR}
	\organisation{Peking University}
	\rorid{02v51f717}
	\address{Haidian District Yiheyuan Road No.5}
	\zip{100871}
	\city{Beijing}
	\country{China}
	\posttext{}
	\affemail{jxchen@bicmr.pku.edu.cn}
	\furtheremail{jxchen2497@gmail.com}}


\begin{abstract}
Length-bounded sweepouts provide a method for bounding the length of the shortest closed geodesic of a closed manifold. In this paper, we generalize this approach to the case of compact 2-dimensional orbifolds homeomorphic to $S^2$ as well as compact 2-dimensional orbifolds with finite orbifold fundamental groups. We establish an inequality for the length of the shortest closed orbifold geodesic in terms of the diameter.
\end{abstract}

\maketitle


\section{Introduction}
\label{sec:introduction}
Given a Riemannian manifold, its closed geodesics are critical points for the energy functional on the space of free loops. Lusternik and Fet \cite{LF52} used this observation together with Morse theory to connect the topology of the free loop space to the existence of closed geodesics. This was later improved by Gromoll, Meyer \cite{GM69},  Vigu\'e-Poirrier, and Sullivan \cite{VS76} to prove the existence of infinitely many geodesics on almost all manifolds. Similar arguments were developed for more quantitative results, for example, on bounding the length of the shortest closed geodesics $\ell(M)$, especially in dimension 2. It is easy to see that given a non-simply connected Riemannian manifold $M$, $\ell(M)\le 2D(M)$,  where $D(M)$ denotes the diameter of $M$. In the harder case of Riemannian 2-sphere, Croke \cite{Cr88} proved that $\ell(M)\le 9D(M)$ and $\ell(M)\le(25+4\sqrt{2})\sqrt{A(M)}$ where $A(M)$ denotes the area of $M$. The first inequality was later improved by Maeda \cite{Ma94} to $\ell(M)\le 5D(M)$. This Morse-theoretic method was generalized by Calabi and Cao \cite{CC92} by noting that closed geodesics appear as critical points of the mass functional in the space of 1-cycles. This new approach allowed Nabutovsky and Rotman \cite{NR02} to improve the previous bounds to $\ell(M)\le 4D(M)$ and $\ell(M)\le 8\sqrt{A(M)}$, which was further improved to $\ell(M)\le 3D(M)$ by Adelstein and Vargas Pallete \cite{adelstein2022length} in the non-negatively curved setting. Croke conjectured that $\ell(M)<(12)^{\frac{1}{4}}\sqrt{A(M)}$. The bound $(12)^{\frac{1}{4}}$ can be achieved by a Riemannian orbifold $S^2_{3,3,3}$ (can be thought of as two equilateral triangles glued together). It is conjectured that this bound can only be achieved in orbifolds.

In the case of Riemannian orbifolds, orbifold geodesics can be defined similarly, and the questions above on the existence and bounds of closed orbifold geodesics can be asked as well. Guruprasad and Haefliger \cite{GH03} defined the orbifold free loop space as an infinite-dimensional orbifold, and they showed that the orbifold topology is related to the existence and the number of closed orbifold geodesics. However, the problem of producing quantitative estimates analogous to those of Nabutovsky and Rotman has not been addressed, and it is the topic of this paper.

\begin{theorem}
	For any compact Riemannian 2-orbifold homeomorphic to $S^2$, denoted by $\mathcal{O}$, $\ell(\mathcal{O})\le 4D(\mathcal{O})$, where $\ell(\mathcal{O})$ is the length of the shortest non-trivial closed orbifold geodesic, and $D(\mathcal{O})$ is the diameter.
\end{theorem}

\begin{corollary}
	For any compact Riemannian 2-orbifold with finite orbifold fundamental group, denoted by $\mathcal{O}$, $\ell(\mathcal{O})\le 8D(\mathcal{O})$.
\end{corollary}

Generalizing to the orbifold case is non-trivial because for a start, we do not have a defined orbifold 1-cycle space that detects closed orbifold geodesics. Also, complications arise from the fact that orbifold geodesics are not locally distance minimizing. The rough idea of the proof of the main result is the following: We define a space of orbifold 1-cycles with 2 segments (orbifold cycles for short) in a way that takes inspiration from Guruprasad and Haefliger's orbifold loop space, and on this space we generalize the Birkhoff curve shortening process and the so-called descent on the steepest direction. Closed orbifold geodesics appear as fixed points of these two processes. Similar to what Nabutovsky and Rotman did in \cite{NR04}, we use these two processes as ingredients and prove that, we can deform 1-parameter families of orbifold cycles shorter than $\ell(\mathcal{O})$ to constant orbifold cycles. Finally, we construct a specific 1-parameter family of orbifold cycles longer than $4D(\mathcal{O})$ that does not deform all the way to zero length. Then the non-contractibility of the family implies $\ell(\mathcal{O})\le 4D(\mathcal{O})$.

The manifold version of the deformation result on 1-cycles is dimension-free and global. In our case, the main issue we encounter is that the 1-cycle space is non-Hausdorff thus creating some complications. We circumvent this problem by controlling the amount of non-Hausdorffness and apply the deformation process only for 1-parameter families of orbifold 1-cycles with 2 segments.

The paper is organized as follows. In Section \ref{sec:preliminaries}, the basic definitions and known results of orbifolds as groupoids, the orbifold loop spaces will be recalled. In Section \ref{sec:cycles}, we define the space of orbifold 1-cycle with 2 segments for later use. In Section \ref{sec:birkhoff}, we generalize the Birkhoff curve shortening process to the space of piecewise-geodesic orbifold 1-cycle with 2 segments and prove its continuity in the Appendix. In Section \ref{sec:descent}, we generalize the ``descent on the steepest direction'' (the descent process in short) to the orbifold case. In Section \ref{sec:deforms}, we combine the Birkhoff process and the descent process to prove two deformation theorems on spaces of cycles. In Section \ref{sec:proof}, we construct the specific family of orbifold cycles and combine it with the deformation results from Section \ref{sec:deforms} to prove the main theorem as well as a corollary.

\section{Preliminaries}
\label{sec:preliminaries}

\subsection{Orbifolds and Groupoids}
\label{subsec:orbifoldsgroupoids}
We will review and set the notations for the concepts we will use throughout
the paper. For deeper investigation, we refer the interested reader to \cite{GH03}*{Section 2}. 

Let $\mathcal{O}$ be a compact Riemannian orbifold and let
$|\mathcal{O}|$ denote its underlying topological space. A Riemannian orbifold chart on $|\mathcal{O}|$ is a $4$-tuple $(X,q,V,\Gamma)$, where $X$ is an open neighborhood in $\mathbb{R}^n$ with $\Gamma$ a finite group by isometry acting on it and $q$ the quotient map by the action, and $V$ is an open neighborhood on $|\mathcal{O}|$ homeomorphic to the quotient space. We choose a Riemannian
orbifold atlas
\[
\{(X_i,q_i,V_i,\Gamma_i)\}_{i\in I},
\]
where the changes of charts are local isometries. Associated to this
atlas is the \'etale groupoid
\[
\mathcal{G} \rightrightarrows X,\qquad X=\bigsqcup_{i\in I} X_i,
\]
whose arrows are germs of changes of charts. We denote the source and
target maps by $\alpha,\omega:\mathcal{G}\to X$, respectively.
Every arrow \(g\in \mathcal{G}\) can be extended to a local isometry on $X$ by \cite{GH03}*{Lemma 2.1.5}.

For \(x\in |\mathcal{O}|\) and a chart $(X,q,V,\Gamma)$ containing $x$, the local group of \(x\) is the isotropy group of any
lift of \(x\) in $\Gamma$; it is well-defined up to conjugacy. A point is
called regular if its local group is trivial, and singular otherwise.

Let \(E\) be a topological
space equipped with an action map \(\alpha_E:E\to X\), a right \(\mathcal{G}\)-action
is a continuous map
\[
E\times_X \mathcal{G}\to E,\qquad (e,g)\mapsto e.g,
\]
where
\[
E\times_X\mathcal{G}=\{(e,g):\alpha_E(e)=\omega(g)\},
\]
satisfying
\[
e.1_{\alpha_E(e)}=e,\qquad \alpha_E(e.g)=\alpha(g),\qquad
(e.g).g'=e.(gg').
\]

We will focus on 2-dimensional orbifolds in this paper. An orientable 2-orbifold is a 2-orbifold with local groups being finite groups by rotation. Except for Corollary \ref{cor:last}, orbifolds are always assumed to be orientable, though we will not use this assumption until Section \ref{subsec:field}.

\subsection{Orbifold Free Loop Space}\label{subsec:loopspace}
We follow the definitions of orbifold loop spaces by 
Guruprasad and Haefliger \cite{GH03}*{Section 2.3}. Let $x$, $y$ be two points of $X$. A continuous $\mathcal{G}$-path $c$ from $x$ to $y$ over a subdivision $0=t_0\le t_1\le t_2\le...\le t_k=1$ of $[0,1]$ is a sequence $(g_0,c_1,g_1,c_2,...,c_k,g_k)$ (sometimes for simplicity denoted by $(c_i,g_i)$) such that

(i) $c_i$ is a continuous map $[t_{i-1},t_{i}]\rightarrow X$ for $i=1,2,...,k$.

(ii) $g_i\in\mathcal{G}$ satisfies $\omega(g_i)=c_i(t_{i})$, $\alpha(g_i)=c_{i+1}(t_{i})$ for all $i$, and $\omega(g_0)=x$, $\alpha(g_k)=y$.

We have the following two operations for $\mathcal{G}$-paths:

(i) (Refining subdivision) Add a point $t^{\prime}$ such that $t_{i-1}\le t^{\prime}\le t_{i}$. We have a new $\mathcal{G}$-path $(g_0,...,g_{i-1},c_i^{\prime},g^{\prime},c_i^{\prime\prime},g_i,...,g_k)$, where the original $c_i$ is replaced by $c_i^{\prime}=c_i|_{[t_{i-1},t^{\prime}]}$ $g^{\prime}=1_{c_i(t^{\prime})}$ and  $c_i^{\prime\prime}=c_i|_{[t^{\prime},t_{i}]}$.

(ii) (Transition of segment) We can change a segment $c_i$ to $c_i^{\prime}$ if there exists a continuous map $h:[t_i,t_{i+1}]\rightarrow \mathcal{G}$ such that $\alpha(h(t))=c_i(t)$ and $\omega(h(t))=c_i^{\prime}(t)$, resulting in a $\mathcal{G}$-path $(g_0,...,g_{i-1}^{\prime},c_{i}^{\prime},g_i^{\prime},...,g_k)$, where $g_{i-1}^{\prime}=g_{i-1}h^{-1}(t_i)$ and $g_i^{\prime}=h(t_{i+1})g_i$.

The equivalence class of $c$ generated by the above two operations, is denoted $[c]_{x,y}$, called a \textit{based orbifold path}. If $x=y$ then we call them \textit{based orbifold loop} and denote by $[c]_x$. The collection of all orbifold based loops will be denoted $\Omega_{x}$, and the union of all $\Omega_{x}$ for $x\in X$ is denoted $\Omega_X$ and called orbifold based space. Note that $\Omega_X$ is not the classical based loop space with one fixed basepoint, but a collection of all based loops. We will be especially interested in the orbifold based loop space of $H^1$-loops. An orbifold based loop $[c]_x$ is $H^1$ if for every representative $c$ of $[c]_x$, each segment $c_i$ is absolutely continuous, and $c$ satisfies that 
$$E(c)=\sum_i^kE(c_i)=\sum_i^k\int_{t_{i-1}}^{t_i}|(c_i)^{\prime}(t)|^2dt<\infty$$
where $E(c)$ is called the energy of $c$, which is simply computed by summing up the energy of all $c_i$'s. The inclusion of the space of orbifold based loop space of $H^1$-loops into the space of orbifold based loop space is a homotopy equivalence, and since in this paper we only care about the former, we will denote it as $\Omega_X$ as well. 

By Guruprasad and Haefliger \cite{GH03}, the based loop space of $H^1$-class $\Omega_X$ has a Riemannian Hilbert manifold structure, defined in the following way: 

(i) Let $H^1(c^*TX)$ be the space of $H^1$-sections on $c^*TX$, where $c^*TX$ is the vector bundle $\sqcup c_i^*TX/\sim$ where the equivalence relationship is defined by gluing together $c_i^*TX$ and $c_{i+1}^*TX$ with $g_i^*$. $T_{[c]_x}\Omega_X$ the \textit{tangent space at $[c]_x$} is the equivalence class of $H^1(c^*TX)$, where the equivalence class is induced by the natural isomorphism between $c^*TX$ and $(c^\prime)^*TX$ where $c$ and $c^{\prime}$ are two $\mathcal{G}$-loop representative of $[c]$. Sometimes we identify $T_{[c]_x}\Omega_X$ and $H^1(c^*TX)$ when there is no confusion.

(ii)Let $[v]$ be any vector in $T_{[c]_x}\Omega_X$ represented by $v=(v_1,...,v_k)$ where $v_i\in c_i^*TX$, and the superscript $\epsilon>0$ is the pointwise bound for $|v_i(t)|$. \textit{ The exponential map} $\exp^{\epsilon}_{[c]_x}:T^{\epsilon}_{[c]_x}\Omega_X\rightarrow \Omega_X$ sends $[v]$ to $[\exp_cv]$, where $\exp_cv=(d_i,\bar{g}_i|_{d_{i+1}(t_i)})$ such that $d_i(t)=\exp_{c_i(t)}v_i(t)$ for all $t$, and $\bar{g}_i$ is a local isometry generated by $g_i$, for all $i$. It can be checked that this map is well-defined and bijective for $\epsilon$ small enough.  It can be checked that $\exp_cH^{1,\epsilon}(c^*TX)$ can be identified with $\exp_{[c]_x}T_{[c]_x}^{\epsilon}\Omega_X$ for $\epsilon$ small. We call $\exp_{[c]_x}T_{[c]_x}^{\epsilon}\Omega_X$ a \textit{modeling neighborhood} on $\Omega_X$ for such small $\epsilon$, for it can be locally modeled by Riemannian Hilbert manifold $H^1(c^*TX)$. We call $\exp_cv$ the \textit{modeling representative} of $[\exp_cv]$ (with respect to $c$).

(iii)For $[v],[w]\in T_{[c]_x}\Omega_X$. The Riemannian metric $(\cdot,\cdot)$ at $[c]_x$ is defined by
$$(v,w)=\sum_{i=1}^{k}\int_{t_{i-1}}^{t_i}\left(\langle v_i(t),w_i(t)\rangle+\langle\frac{D}{dt} v_i(t),\frac{D}{dt} w_i(t)\rangle\right)dt$$

The groupoid $\mathcal{G}$ acts on $\Omega_X$ naturally from the right by permuting the base points. Namely, the action map is defined as $\alpha_{\Omega_X}:\Omega_X\rightarrow X$ sending an orbifold based loop $[c]_x$ to its basepoint $x$, and the action is defined by $\Omega_X\times_X\mathcal{G}\rightarrow \Omega_X$ sending $([c]_x,g)$, with $c=(g_0,c_1,...,c_k,g_k)$ being a representative of $[c]_x$, to $[c^{\prime}]_{\alpha(g)}$, where $c^{\prime}=(g^{-1}g_0,c_1,...,c_k,g_kg)$. The quotient space $\Omega_X/\mathcal{G}$ is denoted  $\Lambda \mathcal{O}$, called the \textit{orbifold free loop space} on $\mathcal{O}$. The element of the free loop space that is represented by the based loop $[c]_x$ will be denoted $[c]$. The orbifold free loop space $\Lambda \mathcal{O}$ is itself an infinite dimensional Riemannian orbifold by \cite{GH03}, with orbifold tangent space $T_{[c]}\Lambda\mathcal{O}$ at $[c]$ being isomorphism classes of $H^{1,\epsilon}(c^*TX)$, exponential map $\exp_{[c]}$ being the map sending $[v]\in H^{1,\epsilon}(c^*TX)$ to $[\exp_{[c]_x}[v]]$ where $[c]_x$ is an orbifold based loop representative of $[c]$ and the equivalence class is quotient by transition of basepoints. $\exp_{[c]}T^{\epsilon}_{[c]}\Lambda\mathcal{O}$ is called a \textit{modeling neighborhood} on $\Lambda\mathcal{O}$ for $\epsilon$ small enough so that $\exp_{[c]}$ is well-defined. 

By \cite{GH03}, it can be checked that $\exp_{[c]}T^{\epsilon}_{[c]}\Lambda\mathcal{O}$ can be identified with $\exp_{[c]_x}T_{[c]_x}\Omega_x/\mathcal{G}_x$, where $\mathcal{G}_x$ is the group in $\mathcal{G}$ that fixes $x$ and the $\mathcal{G}_x$ action is the restriction of the action above that sends $[c]=[(g_0,c_1,...,c_k,g_k)]$ to $[c^{\prime}]=[(g^{-1}g_0,c_1,...,c_k,g_kg)]$. Note that, for $\epsilon$ small, the $\mathcal{G}_x$ action might only be defined for a subgroup of $\mathcal{G}_x$ that fix $\exp_{[c]_x}T_{[c]_x}\Omega_x$. For $v\in H^{1,\epsilon}(c^*TX)$, we say that $\exp_cv$ is a \textit{modeling representative} with respect to $c$ for $\exp_{[c]}[v]$ if $[\exp_cv]=\exp_{[c]}[v]$. Note that modeling representative might not be unique if the $\mathcal{G}_x$ action is non-trivial on $\exp_{[c]_x}T_{[c]_x}\Omega_x$.

The projection of $[c]$ on the underlying topological space is a loop on the underlying topological space $|\mathcal{O}|$, denoted by $|[c]|$, or simply by $|c|$. Denote $|[c]|(t)$ the evaluation of $|[c]|$ at $t$, by $[c](t)$.

Although we have a natural topology induced by the Riemannian metric $(\cdot,\cdot)$ on $\Omega_X$ and $\Lambda\mathcal{O}$, it is not easy to use. Therefore we will instead use the following pointwise topology that corresponds to pointwise convergence.
\begin{defn}\label{def:looppointwisetop}
	$[c^j]\rightarrow[c]$ in the pointwise topology if there exists a number $J$, such that for all $j>J$, $[c^j]$ is in a modeling neighborhood  $\exp_{[c]}^{\epsilon} H^{1,\epsilon}(c^*TX)$ with respect to some representative $c$ of $[c]$, and there exists modeling representatives $c^j$ satisfying $c^j_i(t)\rightarrow c_i(t)$ for all $i$ and for all $t$ as $j\rightarrow \infty$.
\end{defn}

We say an orbifold free loop $[c]$ is \textit{geodesic at a point $t$}$\in(0,1)$ if its representative $c$ is a geodesic near the corresponding point. Note that \textit{geodesic at $0$} (or $1$) will further require that $dg_kdg_0c^{\prime}_1(0)=c^{\prime}_k(1)$, which geometrically means matching velocity at both ends of $0$. Closed orbifold geodesics and piecewise orbifold geodesic free loops with $N$ breaks can be defined in the standard ways (multiplicity of breaks is allowed).

Given two orbifold free paths $[c]$ over the domain $[0,1]$ with representative $(g_0,c_1,...,c_k,g_k)$ and $[d]$ over the domain $[1,2]$ with representative $(h_0,d_1,...,d_l,h_l)$, their concatenation $[c*d]$ can be defined as $[(g_0,...,c_k,g^*,d_1,...,h_l)]$ for $g^*\in\mathcal{G}$ with source $d_1(1)$ and target $c_k(1)$, if their endpoints coincide and the common endpoint is a regular point on the orbifold. The choice of $g^*$ is unique since $[c(1)]=[d(1)]$ is regular.

\section{Orbifold 1-cycles with 2 segments}
\label{sec:cycles}

In the manifold case, Calabi and Cao \cite{CC92} considered the space of 1-cycles with 2 segments:
$$\Gamma=\{(\gamma_1,\gamma_2):\ \gamma_1,\gamma_2:[0,1]\rightarrow M, \ \{\gamma_1(0),\gamma_2(0)\}=\{\gamma_1(1),\gamma_2(1)\}\ \}$$ 
There are three possible types for a 1-cycle $(\gamma_1,\gamma_2)$: 
(i) $\gamma_1(0)=\gamma_1(1)\ne\gamma_2(0)=\gamma_2(1)$, called type I, which is in fact \textit{two loops}.
(ii) $\gamma_1(0)=\gamma_2(1)\ne\gamma_2(0)=\gamma_1(1)$, called type II, which is \textit{one single loop} parametrized over $[0,2]$.
(iii) $\gamma_1(0)=\gamma_1(1)=\gamma_2(0)=\gamma_2(1)$, called type III, which is a \textit{figure ``8''}.

Let $\mathcal{O}$ be an orbifold, and let $(X,\mathcal{G})$ be an induced groupoid structure of $\mathcal{O}$. We will define an orbifold 1-cycle space $\Gamma$ to be the disjoint union of $\Gamma_1$(two loops), $\Gamma_2$(single loops) and $\Gamma_3$(figure ``8''s). We will see that the tangent space as well as the exponential map can be defined on $\Gamma$ similar to the case of orbifold loops but allow a figure ``8'' to exponentiate into the spaces of other two types. These three types of cycle space are ``glued together'' not by identifying a common subspace among them, but by picking a non-Hausdorff topology generated by the image of small balls in the tangent space under the exponential map. The key picture is that this topology allows us to ``flow'' through different types, as we shall see in Section \ref{sec:deforms}. 

\subsection{The Space \texorpdfstring{$\Gamma$}{g} of Orbifold 1-Cycles with 2 Segments}

Let $\Gamma_1$ be the space $(\Lambda\mathcal{O})^2$, corresponding to type I.

Let $\Gamma_2$ be the space $\Lambda\mathcal{O}_{[0,2]}$, corresponding type II, where the subscript $[0,2]$ signifies the domain of parametrization. 

The orbifold 1-cycle space of type III is relatively harder to describe. We will start by defining $\mathcal{G}$-cycles of figure ``8'' on $(X,\mathcal{G})$.

\begin{defn} A \textit{type III $\mathcal{G}$-cycle} over a subdivision $0=t_0\le t_1\le...\le t_{n-1}\le 1=t_n\le...\le t_m=2$ $(0<n<m)$ is defined to be a sequence $c=(g_0,c_1,g_1,...,c_n,g_{n_-},g_{n_+},c_{n+1},...,c_m,g_m)$ for some positive integer $n$ and $m$ such that the first $2n+1$ terms $(g_0,...,g_{n_-})$ (called the first part of $c$) and the rest $(g_{n_+},...,g_m)$ (the second part) are two $\mathcal{G}$-paths based at the same point $\omega(g_0)$.
\end{defn}

The subdivision and transition operations for the two parts of $c$ then induce subdivision and transition of $c$. The equivalence classes generated by the two operations will be called \textit{type III orbifold based 1-cycles}. The equivalence class of $c$ is denoted $[[c]]$. The collection of all such based 1-cycles are denoted $\Omega_3$. $\mathcal{G}$ acts naturally on $\Omega_3$ via the diagonal action by permutation on the shared basepoint of the two parts. We call the equivalence classes of type III orbifold based 1-cycles generated by this action \textit{type III orbifold (free) 1-cycles}, the space of which from now on will be denoted as $\Gamma_3$. 

Note that $\Gamma_1$ and $\Gamma_2$ can also be defined in a similar way as $\Gamma_3$: We call a sequence $c$ consisting of two parts that are both $\mathcal{G}$-path a $\mathcal{G}$-cycle of type I. We define the four \textit{extreme points} of $c$ to be $\omega(g_0)$ $\alpha(g_{n_-})$ $\omega(g_{n_+})$ and $\alpha(g_m)$, denoted by $0^+,1^-,1^+,2^-$. If $0^+$ meets $1^-$ and $1^+$ meets $2^-$, then the equivalence classes generated by subdivision and transition operations from the two parts will be called type I orbifold based 1-cycles, with its collection denoted $\Omega_1$. $\mathcal{G}\times\mathcal{G}$ acts naturally on $\Omega_1$ by basepoint permutation on the two sets of extreme points. $\Gamma_1$ consists of equivalence classes generated by this action. Similar for $\Gamma_2$: A type II $\mathcal{G}$-cycle is a sequence $(g_0,c_1,...,g_{n_-},...,g_m)$ over $0=t_0\le...\le 1=t_n\le...\le t_m$ with $0^+=2^-$ and $1^+=1^-$. A type II based 1-cycle is an equivalence class generated by transition and subdivision. Here transition of $1^+$ is also allowed. The type II orbifold free 1-cycle is the equivalence class generated by permutation on $0^+=2^-$ by $\mathcal{G}$.

Also note that, for any $c$ representing an element in $\Gamma_3$, it also represents respectively an element in $\Gamma_1$ and one in $\Gamma_2$. These inclusions commute with subdivision, transition, and basepoint permutation, thus induce two maps $p_i:\Gamma_3\rightarrow\Gamma_i$ for $i=1,2$. Roughly speaking, $p_1$ tells $\Gamma_3$ to forget that $0^+=1^-$ coincide with $1^+=2^-$, and $p_2$ tells $\Gamma_3$ to forget that $0^+=2^-$ coincide with $1^+=1^-$.

\textit{The space of orbifold 1-cycles with 2 segments $\Gamma$} is defined as $\Gamma_1\sqcup\Gamma_2\sqcup\Gamma_3$. Since we will only talk about orbifold 1-cycle with 2 segments, for simplicity, from now on we will just call them orbifold 1-cycles. The orbifold 1-cycle represented by a $\mathcal{G}$-cycle $c$ will be denoted $[c]$. In this paper, we will only consider $L$-Lipschitz orbifold free 1-cycles for some positive $L$. By $L$-Lipschitz, we mean that for any representative $c$ of $[c]$, any segment $c_i$ is $L$-Lipschitz.

\subsection{Tangent Space and Exponential Map on \texorpdfstring{$\Gamma$}{g}}\label{subsec:modeling}

For $[c]\in\Gamma$ with representative $c=(g_0,...,g_m)$, $c^*TX$ the pullback bundle by $c$ can be defined like in the case of orbifold free loops, though we will have to carefully glue together $c_i^*TX$'s at $\{0,1,2\}$ according to the type of $[c]$:  glue $0^+$ $1^-$ together and $1^+$ $2^-$ together for type I, glue $0^+$ $2^-$ together and $1^+$ $1^-$ together for type II, all four extreme points together for type III. Here for simplicity of notation, we denote the pullback bundle all by $c^*TX$ despite different gluing for different types of $c$. Similar to how orbifold loop space is defined in Section \ref{subsec:loopspace}, for $[c]\in\Gamma_1\sqcup\Gamma_2$, we have the space $H^1(c^*TX)$ of $H^1$ sections on $c^*TX$, respecting the gluing according to type, and the tangent space $T_{[c]}\Gamma$ as an equivalence class of $H^1(c^*TX)$. However for $[c]\in\Gamma_3$, the tangent space is deliberately defined in a slightly different way. 

For $[c]\in\Gamma_3$, we consider a slightly bigger space than $H^1(c^*TX)$, denoted $T_c\Gamma$, allowing possible discontinuity at $t=0,1,2$. Namely, we require $\{v(0^+),v(1^+)\}=\{v(1^-),v(2^-)\}\subset T_{c(0)}X$ for $v\in T_c\Gamma$. Here we are denoting $dg_0v_1(0)$ by $v(0^+)$, $dg_{n_-}^{-1}v_n(1)$ by $v(1^-)$, $dg_{n_+}v_{n+1}(1)$ by $v(1^+)$, and $dg_{m}^{-1}v_m(2)$ by $v(2^-)$. The elements in $T_c\Gamma$ then contain three types: ones with $v(0^+)=v(1^-)$ and $v(1^+)=v(2^-)$, ones with $v(0^+)=v(2^-)$ and $v(1^+)=v(1^-)$, and ones with all four vectors coinciding, whose collection will be denoted by $T^1_c\Gamma$, $T^2_c\Gamma$, and $T^3_c\Gamma$, respectively. Then again similar to Section \ref{subsec:loopspace} the tangent space $T_{[c]}\Gamma$ and $T^i_{[c]}\Gamma$ are defined as an equivalence class of $T_c\Gamma$ and $T^i_c\Gamma$, respectively, for $i=1,2,3$.

The exponential map $\exp_c$ at a $\mathcal{G}$-cycle $c$, the exponential $\exp_{[c]}$, and the modeling neighborhood $\exp_cT^{\epsilon}_c\Gamma$ of $\exp_{[c]}T^{\epsilon}_{[c]}\Gamma$ for $\epsilon$ small can be defined as in the case of orbifold loops spaces in Section \ref{subsec:loopspace} by the exact same argument. 

Denote the collection of $[[c]]$ for $c$ in $S$ a collection of $\mathcal{G}$-cycles by $[[S]]$. The following will be useful. $\exp_{[c]}T^{\epsilon}_{[c]}\Gamma$ can be identified with, for $[c]\in \Gamma_1$  $[[\exp_cT^{\epsilon}_c\Gamma]]/(\mathcal{G}_{0^+}\times\mathcal{G}_{1^+})$ where $\mathcal{G}_{0^+}\times\mathcal{G}_{1^+}$ acts by permuting the two basepoints of the two orbifold based loop of $[[c]]$, for $[c]\in\Gamma_2$  $[[\exp_cT^{\epsilon}_c\Gamma]]/\mathcal{G}_{0^+}$ where $\mathcal{G}_{0^+}$ acts by permuting the basepoints of $[[c]]$, for $[c]\in \Gamma_3$ $[[\exp_cT^{\epsilon}_c\Gamma]]/\mathcal{G}_{0^+}$ where $\mathcal{G}_{0^+}$ acts by permuting the four extreme points of $[[c]]$. Here the $\mathcal{G}_{0^+}$ or $\mathcal{G}_{0^+}\times\mathcal{G}_{1^+}$ action is in fact a subgroup action of $\mathcal{G}_{0^+}$ or $\mathcal{G}_{0^+}\times\mathcal{G}_{1^+}$, respectively, that fixes $[[\exp_cT^{\epsilon}_c\Gamma]]$. We call this subgroup the modeling group.

Note that for $[c]\in\Gamma_3$, exponential of elements in $T_{[c]}^1\Gamma$ actually land in $\Gamma_1$, and similar for $T_{[c]}^2\Gamma$. This deliberately chosen $T_{[c]}\Gamma$ for $[c]\in\Gamma_3$ allows us to exponentiate into type I and type II.

\subsection{The Topology on \texorpdfstring{$\Gamma$}{g}}

The topology on $\Gamma$ will be defined similar to the pointwise topology in Definition \ref{def:looppointwisetop}. 
\begin{defn}
	A sequence $[c^j]\in\Gamma$ converges to $[c]$, if there exists a number $J$ such that for all $j>J$, $[c^j]$ lies in a modeling neighborhood around $[c]$ with respect to representative $c$ of $[c]$, and the modeling representative $c^j$ with respect to $c$ such that their $i$-th segments $c^j_i\rightarrow c_i$ pointwise for all $i$.
\end{defn}

Note that pointwise convergence for Lipschitz curves is the same as uniformly pointwise convergence. Therefore the collection of $\epsilon$-balls $\{[d]\in\exp_{[c]}T^{\epsilon}_{[c]}\Gamma\}$ is a basis for the topology on $\Gamma$ .

This topology is in fact not Hausdorff. For example, for any $[c]\in \Gamma_3$, it can not be separated from $p_1[c]\in \Gamma_1$ by two disjoint open sets. However, the amount of non-Hausdorffness of this topology can be controlled by the following lemma.
\begin{lem}
	For any $[c]\in \Gamma_3$, it cannot be separated from its $\Gamma_1$ counterpart $p_1[c]$ and its $\Gamma_2$ counterpart $p_2[c]$. Also, for any $[c]\ne[d]\in\Gamma_3$ with $p_1[c]=p_1[d]$ or $p_2[c]=p_2[d]$, they cannot be separated from each other. The above are the only cases where $\Gamma$ is non-Hausdorff.
\end{lem}

\begin{proof}
	
	By definition, $\exp_{[c]}$ for $[c]\in\Gamma_i$ for $i=1,2$ only produces orbifold 1-cycle in $\Gamma_i$. Therefore elements in $\Gamma_1$ can be separated from elements in $\Gamma_2$. Also, since $\Gamma_1=\Lambda\mathcal{O}^2$ and $\Gamma_2=\Lambda\mathcal{O}_{[0,2]}$ are Hausdorff, two elements both in $\Gamma_i$ can be separated from each other, for $i=1,2$.
	
	In fact, $\Gamma_3$ is also Hausdorff under the subspace topology: For $[c]\ne[d]\in\Gamma_3$ in the same modeling neighborhood, there exist modeling representatives $c$ $d$ different at least at some $c_i(t)\ne d_i(t)$ for some segment $i$ and some $t\in[0,1]\sqcup[1,2]$. Then take $\epsilon$ to be one-third of the distance from $c_i(t)$ to $d_i(t)$, the $\epsilon$-balls around $[c]$ and $[d]$ separate them. For $[c]$ $[d]$ not in the same modeling neighborhood, suppose that they cannot be separated, that is, for any small $\epsilon$, $\epsilon$-balls around $[c]$ and $[d]$ intersect. Let $[e]$ be in the intersection, then for small enough $\epsilon$, a modeling neighborhood of $[e]$ can be defined that entails $[c]$ and $[d]$.
	
	For $[c]\in\Gamma_3$, an $\epsilon$-ball $B_{\epsilon}([c])$ around $[c]$ intersects $\Gamma_1$ $\Gamma_2$ and $\Gamma_3$ non-trivially. Fix $i=1$ or $2$. It is straightforward to verify that the $\Gamma_i$ component of $B_{\epsilon}([c])$ is exactly $B_{\epsilon}(p_i[c])\backslash p_i(B_{\epsilon}([c])\cap\Gamma_3)$, which is contained by $B_{\epsilon}(p_i[c])$. Therefore $[d]\in\Gamma_i$ with $[d]=p_i[c]$ cannot be separated from $[c]$, since their $\epsilon$-balls always have non-trivial intersection. On the other hand, $[d]\in\Gamma_i$ with $[d]\ne p_i[c]$ can be separated from $[c]$ since $p_i[c]$ can be separated from $[d]$.
	
	For $[c]$ $[d]$ both in $\Gamma_3$ with the same $p_i$ projection for $i=1$ or $2$, they cannot be separated since the $\Gamma_i$ component of their neighborhoods cannot be separated. If they have different $p_1$ and $p_2$ projections, the $\Gamma_i$ component of their neighborhoods can be separated by Hausdorffness of $\Gamma_i$ for $i=1,2,3$.
\end{proof}

The above pointwise convergence in $\Gamma$ clearly induces pointwise convergence on $|\mathcal{O}|$ the underlying space of the orbifold. Namely, if $[c^i]\rightarrow [c]\in \Gamma$ as $i\rightarrow \infty$, then for any $t\in[0,1]\sqcup[1,2]$,  $[c^i](t)\rightarrow [c](t) \in|\mathcal{O}|$.

\subsection{Local \texorpdfstring{$\mathcal{G}$}{G}-homotopy and Homotopy} 
\begin{defn}
	Two $\mathcal{G}$-cycles $c=(g_i,c_i)$ and  $d=(h_i,d_i)$ over a subdivision $0=t_0<t_1<...<t_m=2$ are \textit{locally $\mathcal{G}$-homotopic} if there exists a sequence $H=(l_0,e_1,l_1,...,e_m,l_m)$ where $l_i:[0,1]\rightarrow \mathcal{G}$ are continuous maps such that $l_i(0)=g_i$ $l_i(1)=h_i$, $e_i:[0,1]\times[t_{i-1},t_i]\rightarrow X$ are continuous maps such that $e_i|_0=c_i$ $e_i|_1=d_i$, and $\alpha(l_i(s))=e_{i+1}(s,t_i)$  $\omega(l_i(s))=e_i(s,t_i)$ for $s\in[0,1]$. We denote the homotopy relation between $c$ and $d$ by $c\overset{H}{\sim_{\mathcal{G}}}d$, or simply $c\sim_{\mathcal{G}}d$.
\end{defn}

For any orbifold free cycle $[c]$, any $[d]$ in a modeling neighborhood $\exp_{[c]}^{\epsilon} T^{\epsilon}_{[c]}(\Gamma)$ around $[c]$, can be represented as $\exp_{[c]}^{\epsilon} v=[d]$ for some $v\in T^{\epsilon}_c\Gamma$. Clearly $\exp_{c}( s v)$ for $s\in [0,1]$ is a $\mathcal{G}$-homotopy connecting $c$ and $d$. It is also easy to see that, any two orbifold free (or based) cycles in a common modeling neighborhood have $\mathcal{G}$-homotopic representatives. The notion of $\mathcal{G}$-homotopy can be extended to orbifold free and based loops as well. We say two orbifold 1-cycles $[c]$ and $[d]$ are \textit{homotopic} if there exists a map $H:[0,1]\rightarrow \Gamma$ such that $H(0)=[c]$ and $H(1)=[d]$. It is easy to see that a $\mathcal{G}$-homotopy on representatives induces a homotopy on orbifold 1-cycles.

\subsection{The Length Functional and Stable 1-cycles}

For any $[c]\in\Gamma$ with representative $c$, the \textit{length} of $[c]$ is defined by summing together the length of each $c_i$. The length functional is well-defined but might not be continuous in general in terms of the pointwise topology. However this will not be a problem for us since we will confine ourselves only to piecewise-geodesic orbifold 1-cycles with 2 segments that are $L$-Lipschitz. Denote by $\Gamma^{\le L}$ the set of all orbifold 1-cycles with 2 segments of length less or equal to $L$, $\Gamma^0$ for all orbifold 1-cycles with 2 segments of zero length. 

Given $\mathcal{G}$-cycle $c$, at the four extreme points $0^+$ $1^-$ $1^+$ $2^-$, we can define the so-called ``outward-pointing tangent vectors''. For example, $-dg_0 (c_1)^{\prime}(0)$ would be the outward-pointing tangent vector at $0^+$, denoted $-c^{\prime}(0^+)$ (the negative sign comes from outward-pointing), and $d(g_{n_-})^{-1} (c_n)^{\prime}(1)$ would be the outward-pointing tangent vector at $1^-$, denoted $c^{\prime}(1^-)$. The following definition imitates the definition of stable 1-cycles in manifolds by \cite{NR04}.

\begin{defn}
	An orbifold 1-cycle $[c]$ is stable if, 
	
	(i) $[c]$ is a pair of orbifold closed geodesic if $[c]\in\Gamma_1$.
	
	(ii) $[c]$ is an orbifold closed geodesic over the domain $[0,2]$ if $[c]\in\Gamma_2$.
	
	(iii) $[c]$ is geodesic everywhere except at $t=0,1,2$ and the four outward-pointing tangent vectors for any representative of $c$, after being divided by their norm, sum up to be $0$ if $[c]\in\Gamma_3$.
	
\end{defn}

It is straightforward that a non-trivial orbifold closed geodesic can be constructed from a non-trivial stable orbifold 1-cycle. The argument follows exactly as in the manifold setting \cite{NR04}*{Section 2}. Note that the shortest non-trivial stable orbifold 1-cycle is no shorter than the shortest non-trivial orbifold closed geodesic.

\section{The Birkhoff Curve Shortening Process}
\label{sec:birkhoff}

For the manifold version of the Birkhoff curve shortening process (Birkhoff process in short), we refer the interested reader to \cite{CM11}. In this section, we will define the orbifold version of the Birkhoff process on $\Gamma^{\le L}$, where $\Gamma^{\le L}$ denotes the subspace of $\Gamma$ consisting of $L$-Lipschitz orbifold 1-cycles with two segments with length upper bound $L$ (thus the Birkhoff process on the space of Lipschitz orbifold loops can be defined as its restriction). Note that the Lipschitz constant $L$ and the length bound $L$ are chosen to be the same $L$. We point out here that while the manifold version of the Birkhoff process consists of four different steps, the one presented here only covers the first two, as these will be the only ones we will be using.

\subsection{Birkhoff Process \texorpdfstring{$\Psi$}{p}}\label{subsec:Birkhoff}

For a 2-dimensional compact Riemannian orbifold $\mathcal{O}$, choose an orbifold atlas $\{(X_i,q_i,V_i,\Gamma_i)\}_{i\in I}$ with $X_i$'s being convex metric balls (hence $V_i$'s will be convex metric balls as well) with the property that any two points can be connected via a unique minimizing geodesic segment. This can be achieved by first taking an arbitrary finite orbifold atlas and then restricting to small convex metric balls with the property that their injectivity radius in the ambient space is larger than its radius. Let $\delta$ be the Lebesgue number of $\{V_i\}_{i\in I}$. This $\delta$ could serve as an analogue of ``injectivity radius'' on the orbifold. Consider the corresponding Riemannian groupoid of germs of change of chart $\mathcal{G}\rightrightarrows X$ equipped with a Riemannian structure inherited from the orbifold. Fix a large integer $N$ such that $N\delta>2L$ ($N\delta>L$ is enough for the Birkhoff process, but we need $N\delta>2L$ later for the deformation results in Section \ref{sec:deforms}). This will be the ``break number'' of the Birkhoff process.

The \textit{Birkhoff process} $\Psi$ is a map from $\Gamma^{\le L}$ to $\Gamma^{\le L}$ given in two steps: 

(i) Step One $\Psi^1$ (Reparametrization). Fix a $[c]\in\Gamma^{\le L}$. Let $s_1(t)$ be $\dfrac{\text{Length}[c]|_{[0,t]}}{\text{Length}[c]|_{[0,1]}}$ for $t\in[0,1]$. Then we have a function $t_1:[0,1]\rightarrow[0,1]$ defined by $t_1(s_1):=\inf\{t_1\in[0,1]:s_1(t_1)\ge s_1\}$. $[c]|_{[0,1]}\circ t$ is an orbifold free curve of constant speed on $[0,1]$. The same can be done for $[c]|_{[1,2]}$ which gives us $t_2:[1,2]\rightarrow[1,2]$. Glue together $t_1$ and $t_2$ to get $t:[0,2]\rightarrow[0,2]$. We define $\Psi^1[c]$ as $[c]\circ t$. In short, $\Psi^1[c]$ is $[c]$ reparametrized so that it has constant speed on $[0,1]$ and $[1,2]$ with no change at $t=1$. 

(ii) Step Two $\Psi^2$ (Geodesic Replacement). Fix a $[c]\in \Psi^1\Gamma^{\le L}$. Since Length$[c]|_{[\frac{i}{N},\frac{i+1}{N}]}<\frac{L}{N}<\delta$ for any $i$, we can find a representative $c=(g_i,c_i)$ for $[c]$ over the subdivision $0<\dfrac{1}{N}<\dfrac{2}{N}<...<2$. Let $\bar{c}_i$ be the unique minimizing geodesic connecting $c_i(\frac{i-1}{N})$ and $c_i(\frac{i}{N})$. The piecewise geodesic $\mathcal{G}$-cycle $(g_i,\bar{c}_i)$ represents an orbifold 1-cycle, denoted by $[\bar{c}]$, which is defined to be $\Psi^2[c]$.


One easily checks that in both steps of $\Psi$ the length and the Lipschitz constant are non-increasing. Although the Birkhoff process can be defined for any $H^1$ orbifold 1-cycle
with 2 segments, we are particularly interested in the case of piecewise
geodesic orbifold 1-cycle. Denote piecewise geodesic orbifold 1-cycle with break number $N$ and each geodesic segments shorter than $\delta$ by $\Gamma_N^{\le L}$. Here the break number $N$ means that for any such orbifold 1-cycle $[c]$, orbifold free curves $[c]|_{[0,1]}$ and $[c]|_{[1,2]}$ each has at most $N-1$ non-geodesic points.

\subsection{Birkhoff Homotopy}
For each step $\Psi^i$ of the Birkhoff shortening process and each $[c]$ in the domain of $\Psi^i$, we can define a continuous homotopy $\Phi^i$ from $[0,1]\times\Gamma^{\le L}_N$ to $\Gamma^{\le L}_{3N}$ such that $\Phi^i(0,[c])=[c]$ and $\Phi^i(1,[c])=\Psi^i[c]$ for $i=1,2$. Note that the break number will have to increase from $N$ to $3N$ since more breakpoints will be introduced in our construction.

(i) Step One $\Phi^1$. Let $[c]\in\Gamma^{\le L}_N$. There exists a unique piecewise-linear non-decreasing map $P_{[c]}:[0,2]\rightarrow[0,2]$ such that $[c]=\Psi^1[c]\circ P_{[c]}$. We denote by $f_s^{[c]}$ a function from $[0,2]$ to $[0,2]$ sending $t$ to $st+(1-s)P_{[c]}(t)$. $\Phi^1(s,[c])$ is defined as $\Psi^1[c]\circ f^{[c]}_s$.

(ii) Step Two $\Phi^2$. Let $c=(c_i,g_i)$ be a representative of a constant-speed piecewise-orbifold-geodesic orbifold 1-cycle $[c]$ over the subdivision $0<\dfrac{1}{N}<\dfrac{2}{N}<...<2$. For simplicity, set $t_i=\frac{i}{N}$ and $t_i^s=st_{i-1}+(1-s)t_i$ for all $i$ and $s\in[0,1]$. We introduce $c^s_i:[t_{i-1},t_i]\rightarrow X_i$ to be the concatenation of the minimizing geodesic $\bar{c}^s_i$ between $c_i(t_{i-1})$ and $c_i(t_i^s)$ and the original $c_i$ on $[t^s_i,t_i]$. We define $\Phi^2_{\mathcal{G}}(s,c)$ as the $\mathcal{G}$-cycle $(c_i^s,g_i)$ for $s\in[0,1]$. Set $\Phi^2(s,[c]):=[\Phi^2_{\mathcal{G}}(s,c)]$. Note that $\Phi^2_{\mathcal{G}}$ is a local $\mathcal{G}$-homotopy between $c$ and $\Psi^2(c)$. Hence $\Phi^2$ is a homotopy between $[c]$ and $\Psi^2[c]$.

It is immediate from the definition that in both homotopies, the length and the Lipschitz constant are non-increasing on $s$. The \textit{Birkhoff homotopy} $\Phi:[0,2]\times\Gamma_N^{\le L}\rightarrow\Gamma_{3N}^{\le L}$ is then defined by gluing the two homotopies. The continuity of the Birkhoff homotopy is left to Appendix A. This check is not hard but tedious because we need to match common subdivision in order to compare orbifold cycles in a common modeling neighborhood. 

\section{Descent on the Steepest Direction}
\label{sec:descent}

The manifold version of the ``Descent on the Steepest Direction'' (the descent process for short) for piecewise-geodesic 1-cycles can be looked up on \cite{NR04}*{Lemma 3}.

The idea of the descent process is simple: For a piecewise-geodesic 1-cycle, at each breakpoint, we compute the sum of all outward pointing unit tangent vectors originated from this breakpoint to get a vector, called the descent vector at this breakpoint. We flow the endpoints along their corresponding descent vectors for a short period of time, then rejoin the new endpoints with unique minimizing geodesics to get a new piecewise-geodesic 1-cycle (as shown in figure \ref{fig:descent}). 

\begin{figure}
	\centering
	\includegraphics[scale=0.2]{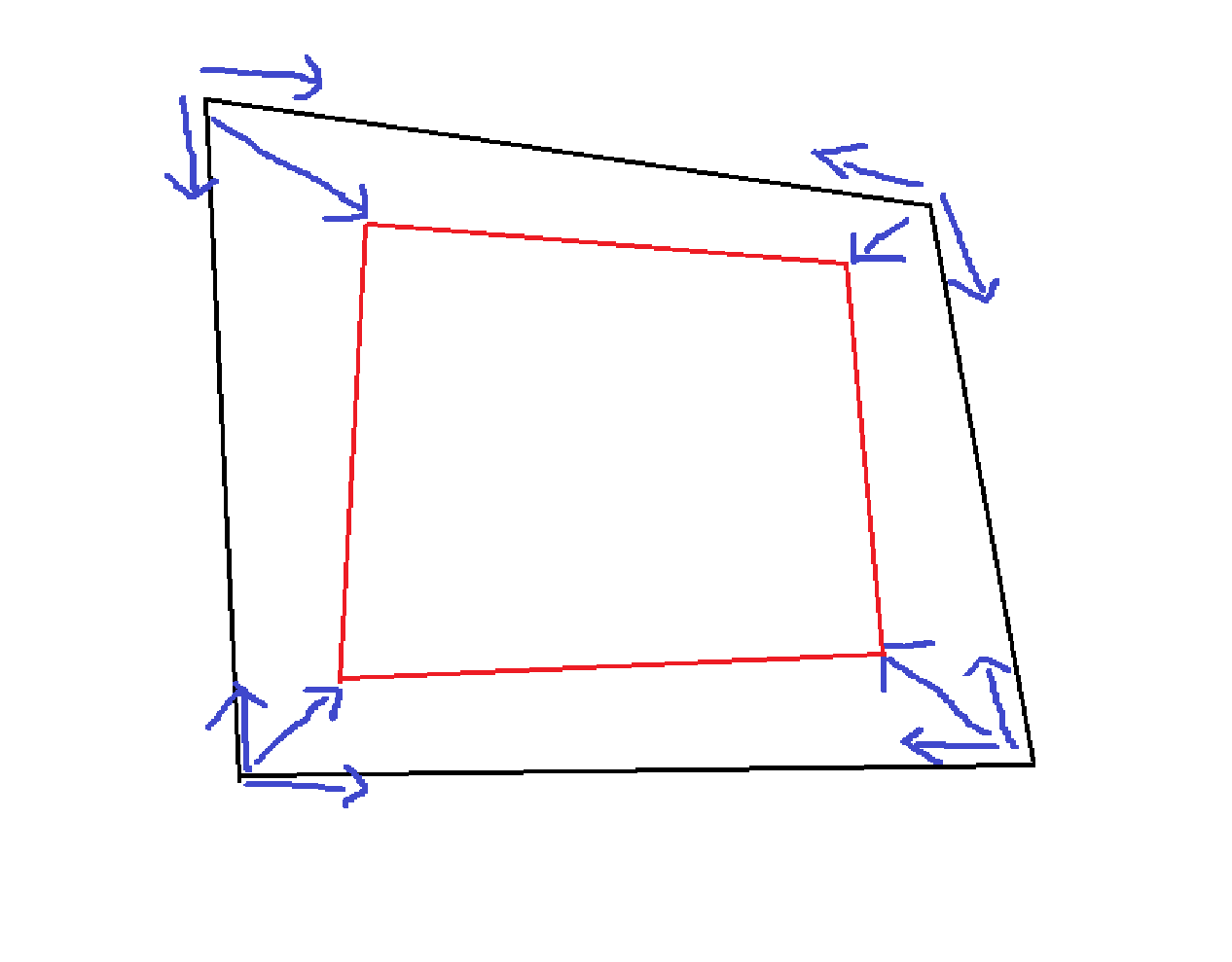}
	\caption{Descent Process}
	\label{fig:descent}
\end{figure}


\subsection{Ordered 1-Cycle Spaces}
In the last section, $\Gamma^{\le L}_N$ is defined as the space of $L$-Lipschitz piecewise geodesic orbifold 1-cycle with 2 segments with break number $N$, with length no longer than $L$, and each geodesic segment no greater than $\delta$. In order to define the descent process, we need a space of ``ordered'' piecewise geodesic orbifold 1-cycle with break number $N$ (with length bound and Lipschitz constant $L$), denoted by $\hat{\Gamma}_N^{\le L}$, which is defined as a subspace of $\Gamma_N^{\le L}\times [0,2]^{2N}$ consisting of element $([c],t_1,...,t_{2N})$, denoted by $[c]_o$, where $[c]\in\Gamma_N^{\le L}$ and $0=t_0\le t_1\le t_2\le ...\le t_{N}=1\le...\le t_{2N}=2$ is a choice of breakpoints of $[c]$, called the ordered subdivision of $[c]_o$, where ``o'' stands for ``ordered''. $\mathcal{G}$-cycle $c$ over the ordered subdivision of $[c]_o$ is called the ordered representative of $[c]_o$, if $c$ is a representative of $[c]$, for $[c]_o\in\hat{\Gamma}_N^{\le L}$. This space adopts the subspace topology of the product topology on $\Gamma_N\times [0,2]^{2N}$.


There might be multiple $[c]_o\in\hat{\Gamma}^{\le L}_N$ corresponding to one $[c]\in\Gamma^{\le L}_N$. This happens when some $t_i$ in the ordered subdivision of $[c]_o$ has multiplicity, or $[c]$ has ``fake'' breakpoints, in the sense that, for the ordered representative $c$ over the subdivision $t_0\le...\le t_{2N}$ with each $c_i$ being geodesic, some breakpoints $t=t_i$ is in fact geodesic.

Note that in \cite{NR04}, the notion of ``ordered'' 1-cycles is also implicitly needed. From now on, we will mostly focus on ordered piecewise-geodesic orbifold 1-cycles. The Birkhoff homotopy previously defined on the unordered orbifold 1-cycles admits an ``ordered'' version keeping track of breakpoints. We omit the explicit formula.

The following subspace is also needed. $g^{\le L}_N$ is the subspace of $\hat{\Gamma}^{\le L}_N$ such that each geodesic segment is no longer than $\frac{\delta}{2}$. We will only define the descent process on $g^{\le L}_N$. 

\subsection{Descent Vector}
Let $[c]_o$ be in $g^{\le L}_N$ with an ordered representative $c=(c_i,g_i)$ over the ordered subdivision $0=t_0<...< t_{2N}=2$. Each geodesic segment $c_i$ contributes two tangent vectors $v_{i-1}^+:=c_i^{\prime}(t_{i-1})$ and $v_i^-:=-c_i^{\prime}(t_i)$ at the endpoints, which after normalized will be called the unit outward-pointing tangent vectors. A descent vector at a basepoint is defined by summing all the ``translated'' unit outward-pointing tangent vector ``associated'' to this basepoint. We unpack ``translated'' and ``associated'' in the following.

Consider the equivalence relation ``merged'' on basepoints generated by the following:
(i) Basepoints $c_i(t_{i-1})$ and $c_i(t_i)$ are merged for $i$ if $c_i$ is constant. 
(ii) $c_i(t_i)$ and $c_{i+1}(t_i)$ are merged for all $i$ except for $i=N$. 
(iii) $c_1(t_0)$ is merged to $c_N(t_{N})$, $c_{N+1}(t_{N})$ is merged to $c_{2N}(t_{2N})$ if $[c]\in \Gamma_1$.
(iv) $c_1(t_0)$ is merged to $c_{2N}(t_{2N})$, $c_{N+1}(t_{N})$ is merged to $c_N(t_{N})$ if $[c]\in \Gamma_2$.
(v) all four basepoints in (iv) are all merged if $[c]\in\Gamma_3$.

A unit outward-pointing tangent vector (o-vector for short) is said to be ``associated'' to a basepoint, if its underlying basepoint is ``merged'' to it. An o-vector at a basepoint defines a ``translated'' o-vector at any associated basepoint, via the following ``translation'' maps along with their inverse and compositions:
(i) For $i$ with $c_i$ being constant, the translation map from $T_{c_i(t_{i-1})}X$ to $T_{c_i(t_i)}X$ is an identity.
(ii) Translation map $dg_i:T_{c_{i+1}(t_i)}X\rightarrow T_{c_i(t_i)}X$ for all $i$ except for $i=N$.
(iii) Translation map $d(g_{N_-}\circ g_0):T_{c_1(t_0)}X\rightarrow T_{c_N(t_{N_-})}X$ and $d(g_{2N}\circ g_{N_+}):T_{c_{N+1}(t_{N_+})}X\rightarrow T_{c_{2N}(t_{2N})}X$ if $[c]\in \Gamma_1$.
(iv) Translation map $d(g_{2N}\circ g_0):T_{c_1(t_0)}X\rightarrow T_{c_{2N}(t_{2N})}X$ and $d(g_{N_-}\circ g_{N_+}):T_{c_{N+1}(t_{N_+})}X\rightarrow T_{c_{N}(t_{N_-})}X$ if $[c]\in \Gamma_2$.
(v) All four translation map in (iii) and (iv) if $[c]\in \Gamma_3$.

The geometric picture of the orbifold descent vector is exactly the same as the manifold version. The ``translation'' step is needed since a $\mathcal{G}$-cycle is not really connected, rather it is connected via $g_i$'s. Denote the descent vector at $c_i(t_{i-1})$ by $v(c)_{i-1}^+$ and the one at $c_i(t_i)$ by $v(c)_i^-$ for all $i$. Collect the descent vectors and denote $v(c):=(v(c)_0^+,v(c)_1^-,...,v(c)_{2N}^-)$, called the descent vector for $c$.  The natural isomorphism between $T_c\Gamma$ and $T_{c^{\prime}}\Gamma$ induces natural isomorphism between $v(c)$ and $v(c^{\prime})$ for $c$ and $c^{\prime}$ different ordered representative of $[c]_o$ with geodesic segments, and denote the equivalence class $v[c]_o$. $v(c)$ defines a variation $c^s=(c_i^{s},g_i^{s})$ for $s$ small, where $c_i^s$ is the minimizing geodesic between $\exp sv(c)_{i-1}^+$ and $\exp sv(c)_i^-$. One can check that the family $[c^s]_o$ is well-defined for $v[c]_o$ for $s>0$ small, and in fact contains piecewise geodesic orbifold 1-cycles of the same type as $[c]_o$.

As in the manifold case, for any $[d]_o$ close enough to $[c]_o$, the subdivision of endpoints of $[d]_o$ into merged classes is the same or finer than that of $[c]_o$. We say that $[c]_o$ is of a higher type than $[d]_o$, if the subdivision of endpoints of $[d]_o$ into merged classes is finer than that of $[c]_o$.

\subsection{The First Variation of Length}\label{subsec:variation}
Let the Length of $c$ be the sum of the length of all its segments. Let $v=(v_0^+,v_1^-,v_1^+,...,v_{2N-1}^+,v_{2N}^-)$ be a sequence of vectors where $v_i^+$ and $v_i^-$ are located at $c_{i+1}(t_{i})$ and $c_i(t_i)$ respectively for all $i$. This type of $v$ can be thought of as a direction for variation of $c$, constructed by flowing each $c_{i+1}(t_{i})$ and $c_i(t_i)$ along $v_i^+$ and $v_i^-$ respectively for a small period $s$, then rejoin by unique minimizing geodesic $c^s_i$ between them. Note that if no further constraints are put on $v$, then the variation by $v$ might not be a $\mathcal{G}$-cycle anymore. One can easily check that the variation by the descent $v(c)$ constructed from the previous section, at least for a short distance, produces $\mathcal{G}$-cycles, which in fact are of the same type as $c$. 

The first variation of length on $v(c)$ can be computed as follows:

\begin{align*}\frac{\partial L}{\partial v(c)}=&\frac{d}{ds}|_{s=0}L(c^s)=\sum_i\frac{d}{ds}|_{s=0}L(c^s_i)\\
	=&\sum_i\langle -\dfrac{c_i^{\prime}(t_{i-1})}{|c_i^{\prime}(t_{i-1})|}, v(c)^+_{i-1}\rangle+\langle \dfrac{c_i^{\prime}(t_i)}{|c_i^{\prime}(t_i)|},v(c)^-_i\rangle
\end{align*}

In the last sum above, $-\frac{c_i^{\prime}(t_{i-1})}{|c_i^{\prime}(t_{i-1})|}$ and $\frac{c_i^{\prime}(t_i)}{|c_i^{\prime}(t_i)|}$ are exactly the negative of the o-vector for all $i$. Let $J=\{J_1,...,J_K\}$ be the subdivision of the collection of all basepoint $\{c_1(0),c_1(t_1),...,c_{2N}(2)\}$ into merged equivalence classes. Rewrite the sum as $\sum_{k=1}^K\sum_{p\in J_k} -\langle v(c)_p,v_p\rangle$ where $v_p$ is the o-vector at $p$. Note that translation by groupoid elements does not change the inner product, we could translate all terms $-\langle v(c)_p,v_p\rangle$ for $p\in J_k$ to one specific $p_k\in J_k$, then the sum $\sum_{p\in J_k} -\langle v(c)_p,v_p \rangle$ equals $-\langle v(c)_{p_k},\sum_{p\in J_k} \tilde{v}_p\rangle$, where $\tilde{v}_p$ is the translated o-vector of o-vector $v_p$ from $p$ to $p_k$. Note that the sum in the inner product gives exactly $v(c)_{p_k}$, therefore the first variation of length gives $\dfrac{\partial L}{\partial v(c)}=-\sum_{k=1}^K||v(c)_{p_k}||^2\le 0$, therefore length non-increasing, and the equality is attained if and only if $v(c)$ is trivial. It is straightforward to check that $\frac{\partial L}{\partial v[c]_o}[c]_o$ is well-defined. $-\frac{\partial L}{\partial v(c)}$ ($\ge0$) is said to be the norm of $v[c]_o$. The norm of $v[c]$ vanishes if and only if the sum of all outward-pointing unit tangent vectors for all clusters are $0$, if and only if $[c]$ is a stable 1-cycle.

\subsection{The Descent Vector Fields}\label{subsec:field}

In this section, for each $[c]_o\in g^{\le L}_N\backslash \hat{\Gamma}^0$, we construct from $v[c]_o$ a descent vector field $V[c]_o$ on a neighborhood around $[c]_o$ in $\hat{\Gamma}^{\le L}_N$. 

Let $c$ be an ordered representative of $[c]_o$, and $v(c)$ the descent vector of $c$. Let $X_i$ be the component of $X$ where $c_i$ lies. The descent vectors at breakpoints $v(c)_{i-1}^+$ and $v(c)_i^-$ extend to vector fields $V(c)_{i-1}^+$ and $V(c)_i^-$ on $X_i$ by parallel translation along minimal geodesic for any $i$. Let $V(c)$ be $(V(c)_0^+,V(c)_1^-,...,V(c)_{2N}^-)$. Let $[d]_o$ be close to $[c]_o$, with modeling representative $d$ in a modeling neighborhood around $[c]$ with respect to $c$. Since $d$ is a modeling representative with respect to $c$, it has the ordered subdivision of $[c]_o$ rather than that of $[d]_o$. However by transition operation on the difference between the two subdivisions, we could obtain a new representative that is an ordered representative of $[d]_o$, still denoted as $d$. Then $[d]_o\rightarrow [c]_o$ if and only if $d=(d_i,h_i)$ with ordered subdivision $r_0\le...\le r_{2N}$ satisfies that $d_i\rightarrow c_i$, $h_i\rightarrow g_i$, $r_j\rightarrow t_j$, for any $i$ and $j$. Let $d^s_{i-1,+}$ be the exponential of $V(c)_{i-1}^+$ at $d_i(r_{i-1})$ for time $s$, and $d^s_{i,-}$ be the exponential of $V(c)_i^-$ at $d_i(r_i)$ for time $s$. Let $d^s$ be $(d^s_i,h_i^s)$ where $d^s_i$ is the minimizing geodesic connecting $d^s_{i-1,+}$ and $d^s_{i,-}$, and $h^s_j$ be the germ of the extension of $g_j$ at the new endpoints, for all $i$ and $j$. 

One easily checks that the subdivision of endpoints of $d$ into merged classes is finer than that of $c$, if $[d]_o$ is close enough to $[c]_o$. Therefore the descent vector fields in $V(c)$ flow merged endpoints of $d$ to merged endpoints of $d^s$ for any small $s$, which implies $d^s$ is still a $\mathcal{G}$-cycle and is of the same type as $d$. Also, the linear combination of descent vector fields also flows $\mathcal{G}$-cycle to $\mathcal{G}$-cycle of the same type.

Let $V(c)_d$ be $(V(c)|_{d_1(r_0)},V(c)|_{d_1(r_1)},...,V(c)|_{d_{2N}(r_{2N})})$ the system of $2N$ tangent vectors on $X$ we used to flow $d$. $V(c)_d$ corresponds to a piecewise Jacobi field on $d$, which corresponds to a tangent vector in $T_d\Gamma$, still denoted by $V(c)_d$. $T_d\Gamma$ can be identified as a neighborhood of based orbifold 1-cycles $[[\exp_dT^{\epsilon}_d\Gamma]]$. Recall from Section \ref{subsec:modeling}, $[[\exp_dT^{\epsilon}_d\Gamma]]$ after quotient by the modeling group can be identified with $\exp_{[d]}^{\epsilon}T_{[d]}\Gamma$. In other words, $V(c)_d$ needs to be modeling group invariant so that $[V(c)_d]$ $[d^s]$ is well-defined.

\begin{lem}
	$V(c)_d$ is invariant under the modeling group action.
\end{lem}
\begin{proof}
	We will prove for $[d]\in\Gamma_1$, same method applies for $[d]$ of other types. This is where we use the orientability of $\mathcal{O}$. In such 2-orbifold, isotropy groups in $\mathcal{O}$ are finite group by rotations.
	
	Suppose $[d]|_{[0,1]}$ and $[d]|_{[1,2]}$ are both non-constant, with a point $[d](\tau)$ for some $\tau\in[1,2]$ with a distance more than $K\epsilon$ away from the basepoint of $[d]|_{[0,1]}$. Here $K$ is chosen to be the inverse of $\sin(\frac{2\pi}{3|\mathcal{G}_{0^+}|})$. Assume $\tau$ is in the domain of $h_1$, if not, we could switch the representative to one with $\tau$ in $h_1$, since $K\epsilon$ can be chosen to be smaller than $\delta$. Let $d^{\prime}=(d^{\prime}_i,h_i^{\prime})$ be a modeling representative of $[d^{\prime}]\in\exp_{[d]}T^{\epsilon}_{[d]}\Gamma$. The basepoint permutation action by $(g_1,g_2)\in\mathcal{G}_{0^+}\times\mathcal{G}_{1^+}$ on $\exp_dT^{\epsilon}_d\Gamma$ maps $d^{\prime}=(h^{\prime}_0,d^{\prime}_1,...,h^{\prime}_{N^-},h^{\prime}_{N_+},...,h^{\prime}_{2N})$ to $(g_1,g_2).h^{\prime}=(g_1\circ h^{\prime}_0,d^{\prime}_1,...,h^{\prime}_{N^-}\circ g_1^{-1},g_2\circ h^{\prime}_{N_+},...,h^{\prime}_{2N}\circ g_2^{-1})$, where $g_1$ and $g_2$ are treated as their local isometry extension. We will prove that the only subgroup of $\mathcal{G}_{0^+}\times\mathcal{G}_{1^+}$ that fixes  $[[\exp_dT^{\epsilon}_d\Gamma]]$ is $\{e\}$. If not, $[[(g_1,g_2).h^{\prime}]]\in [[\exp_dT^{\epsilon}_d\Gamma]]$ for some non-trivial $(g_1,g_2)$, then we can find another representative for $[[(g_1,g_2).h^{\prime}]]$ with segment $h^{\prime\prime}$ in $\exp_dT^{\epsilon}_d\Gamma$. If $\epsilon$ is small, then $h^{\prime\prime}_i=h^{\prime}_i$ for all $i$ by discreteness of similar groupoid element in \'etale groupoid (two groupoid elements are similar if there is groupoid element connected their sources). If the component $X_1$ where $d_1$ lies is flat, it is an elementary Euclidean geometry exercise to show that the distance from $h^{\prime\prime}_1(\tau)$ to $h_1^{\prime}(\tau)$ is more than $\epsilon$ by our choice of $K$. In the non-Euclidean case, by smallness of $\epsilon$, the length distortion is minimal. Therefore $g_1$ can only be the identity $e\in\mathcal{G}_{0^+}$. The same argument applies for $g_2$. Which contradicts the assumption that $(g_1,g_2)$ is non-trivial. Therefore, in this case $V(c)_d$ is modeling group invariant.
	
	Now suppose $[d]|_{[0,1]}$ or $[d]|_{[1,2]}$ is constant. WLOG, let it be the former. Since the subdivision of endpoints of $[d]_o$ is finer than that of $[c]_o$, $[d]|_{[0,1]}$ can only be constant if $[c]|_{[0,1]}$ is constant, in which case $v(c)|_{[0,1]}$ vanishes everywhere. The also vanishing $V(c)|_{[0,1]}$ is invariant under $\mathcal{G}_{0^+}$, since a vanishing vector field is invariant under any group action. The same argument applies if $[d]|_{[1,2]}$ is also constant. If it is not constant, then the argument in the previous paragraph forces $\mathcal{G}_{1^+}$ to be trivial.
\end{proof}

Define $V[c]_o|_{[d]_o}$ as the tangent vector $([V(c)|_d],0,...,0)$ on $\hat{\Gamma}^{\le L}_N$. Here the $0$'s mean that the ordered breakpoints $r_i$'s of $[d]_o$ are unchanged. The flow of $V[c]_o$ does not exceed $\Gamma^{\le L}_N$ for a short period of time since it starts from $g^{\le L}_N$. 


\subsection{Properties of Descent Process}


Let $v=(v_i)_{\{i=0^+,1^-,...,2N^-\}}$ be a vector system which contains one vector at each endpoint of ordered piecewise geodesic orbifold 1-cycle. We also label the descent vector $v(c)=\{v(c)_i\}$ as above. Recall the first variation of length formula for $c$ on the direction of vector sequence $v=(v_i)$ ($i=0^+,...,2N^-$):

$$\dfrac{\partial L}{\partial v}=\sum_i^{2N}\langle v_{(i-1)^+} ,e_{i,1}\rangle + \langle v_{i^-} ,e_{i,2}\rangle$$

Here denote by $e_{i,1}$ and $e_{i,2}$ the negative of the o-vector at the left endpoint and the right endpoint, respectively, of the segment $c_i$ for all $i$. We will compute $\dfrac{\partial L}{\partial v}$ for $d$ where $v=V(c)|_{d}$. 

\begin{lem}\label{lem:continuitylength}
	$\frac{\partial L}{\partial V(c)|_{d}}\rightarrow \frac{\partial L}{\partial v(c)}$ as $[d]_o\rightarrow [c]_o$. 
\end{lem}
\begin{proof}
	WLOG, assume $[d]_o$ is of no higher type than $[c]_o$. In this case, we might have one of the two cases for $i$'s: (i) $c_i$ of $c$ is constant while $d_i$ in $d$ is not, (ii) and $c_i$ and $d_i$ are both constant or both non-constant. For $i$ with $c_i$ being constant while $d_i$ is not, the corresponding term in the sum is $\langle V(c)_{i^-}|_{d_{i^-}},e_{i,2}(d)\rangle+\langle V(c)_{(i-1)^+}|_{d_{(i-1)^+}},e_{i,1}(d)\rangle$. Note that $V(c)_{i^-}=V(c)_{(i-1)^+}$ since $c_i$ is constant. The extra non-zero term converges to $0$ by smoothness of $V(c)_{i^-}$ and smallness of $d_i$. On the other hand, for $i$ with $c_i$ and $d_i$ are both constant or both non-constant, we have $V(c)_i|_{d}\rightarrow v(c)_i$, $e_{i,1}(d)\rightarrow e_{i,1}(c)$, and $e_{i,2}(d)\rightarrow e_{i,2}(c)$.
\end{proof}

Recall that $||v[c]_o||=0$ if and only if $[c]$ is stable. Under the assumption that $L$ is less than $\ell(\mathcal{O})$, $||v[c]_o||=0$ for $[c]_o\in\hat{\Gamma}^{\le L}_N$ if and only if $[c]_o$ is of zero length. Combining that with Lemma \ref{lem:continuitylength}, we immediately have the following.

\begin{lem}\label{lem:fields}
	Fix any $\epsilon$. For any $[d]_o\in g^{\le L}_N\backslash g^{<\epsilon}_N$, there exists a neighborhood $W_{[d]_o}$ such that the descent vector field $V[d]_o$ can be defined on $W_{[d]_o}\subset \hat{\Gamma}^{\le L}_N$, and for any $[c]_o\in W_{[d]_o}$, $2||v[d]_o||>-\frac{\partial L}{\partial V[d]_o}|_{[c]_o}>\frac{||v[d]_o||}{2}$. 
\end{lem}

In fact, up to restricting the neighborhood $W_{[d]_o}$ in the above Lemma, we can require that for $[d]_o$ and for any $[c]_o\in W_{[d]_o}$, flowing $[c]_o$ along any normalized vector field in $\hat{\Gamma}^{\le L}_N$ for a fixed amount of time $s$ does not exceed $\hat{\Gamma}^{\le L}_N$, where the norm of a vector on $\hat{\Gamma}^{\le L}_N$ is computed via the absolute value of its first variation of length. $s$ can simply be chosen as $\frac{\delta}{2}$, since within this timeframe no geodesic segment can go from shorter than $\frac{\delta}{2}$ (which corresponds to requirement on elements in $g^{\le L}_N$) to longer than $\delta$.

\begin{lem}\label{lem:finite}
	A finite cover of $g^{\le L}_N\backslash g^{<\epsilon}_N$ can be chosen from the collection of all $W_{[d]_o}$ for $[d]_o\in g^{\le L}_N\backslash g^{<\epsilon}_N$. 
\end{lem}

\begin{proof}
	Recall that in $X=\sqcup X_i$, $X_i$ is the modeling (manifold) neighborhood in an orbifold chart of $\mathcal{O}$ for any $i$. Expand each orbifold chart by a bit so that the modeling neighborhood in the new chart contains the closure $\bar{X}_i$ for all $i$. Let $\bar{X}$ be $\sqcup \bar{X}_i$ and $\bar{\mathcal{G}}$ be the groupoid associated to the new atlas restricted to objects in $\bar{X}$. By compactness of $\mathcal{O}$ and finiteness of our choice of atlas, $\bar{\mathcal{G}}$ and $\bar{X}$ are compact. For $[c]_o\in \hat{\Gamma}^{\le L}_N$ with an ordered representative $c$, we identify $c=(g_0,c_1,...,g_{2N})$ with an element in the compact finite-dimensional space $S:=(\bar{\mathcal{G}})^{2N+2}\times \bar{X}^{4N}\times[0,2]^{2N}$, by sending $c=(g_0,c_1,...,g_{2N})$ with ordered subdivision $0=t_0\le...\le t_{2N}$ to $(g_0,g_1,...,g_{2N},c_1(t_0),c_1(t_1),c_2(t_1),...,c_{2N}(t_{2N}),t_0,...,t_{2N})$. On the other hand, for an element in $S$, if the components $\bar{\mathcal{G}}^{2N+2}$, $\bar{X}^{4N}$, and $[0,1]^{2N}$ are carefully matched, it can correspond to an ordered piecewise geodesic orbifold 1-cycle by rejoining the points in $\bar{X}^{4N}$ pairwise via minimizing geodesic. Denote the collection of such elements in $S$ by $S^{\prime}$. One easily checks that $S^{\prime}$ is closed, hence compact in $S$. The problem is that an element in $S^{\prime}$ could correspond to an orbifold 1-cycle in $\Gamma^i$ for $i=1,2,3$. To fix this, we consider $S^{\prime\prime}=S^{\prime}\sqcup S^{\prime}\sqcup S^{\prime}$, from which a map to $\hat{\Gamma}^{\le L}_N$ can be defined.
	
	For the collection $\{W_{[d]_o}\}$, $W_{[d]}$ is the exponential of a subset of a modeling neighborhood. For any ordered representative $d$ of $[d]_o$, let $W_d$ be the subset of the modeling neighborhood $\exp_d T^{\epsilon_d}_dT_d\Gamma$ that exponentiate to $W_{[d]_o}$. $\{W_d\}$ is an open cover of $S^{\prime\prime}$, which admits finite subcover, which corresponds to a finite subcover of $\{W_{[d]_o}\}$.
	 
\end{proof}

\section{Deformations Results}
\label{sec:deforms}

\subsection{Homotopy from \texorpdfstring{$\hat{\Gamma}^{\le L}_N$}{gLN} to \texorpdfstring{$\hat{\Gamma}^{\le \epsilon}_N$}{geN}}
\label{subsec:deform1}

The following deformation result is inspired by \cite{NR04}*{Lemma 3}. The idea is to apply the Birkhoff process and the descent process alternatively. 

\begin{thm}\label{thm:deform1}
	Let $\mathcal{O}$ be a compact 2-orbifold homeomorphic to $S^2$. Let $L$ be a positive number less than the length of the shortest non-trivial orbifold closed geodesic on $\mathcal{O}$. Let $f:[0,1]\rightarrow \hat{\Gamma}^{\le L}_N$ be a continuous map, then for any $\epsilon>0$, a continuous homotopy $H:[0,1]\times[0,1]\rightarrow \hat{\Gamma}^{\le L}_{3N}$ can be constructed, such that $H(\cdot,0)=f$ and $H(\cdot,1):[0,1]\rightarrow \hat{\Gamma}^{<\epsilon}_N$.
\end{thm} 

Here we identify an ordered piecewise geodesic orbifold 1-cycle $[c]_o$ in $\hat{\Gamma}^{\le L}_N$ with a $[d]_o$ in $\hat{\Gamma}^{\le L}_{3N}$, if $[c]$ and $[d]$ are the same orbifold 1-cycle and the breakpoints of $[d]$ are exactly three copies of the break points of $[c]$.

\begin{proof}
	Start with a step of Birkhoff homotopy of break number $N$ so that $f$ is homotoped into $g^{\le L}_N$. For simplicity of notation, still denote the time-one endpoint of this homotopy $f$.
	
	Fix an $\epsilon>0$. Let orbifold 1-cycle $[d^j]_o$, descent vector fields $V[d^j]_o$ $V_j$, domain $W_j$ for all $j\in J$, and $s$ be as in Lemma \ref{lem:fields}, Lemma \ref{lem:finite} and the paragraph between them, corresponding to $\epsilon$, where $J$ is a finite index set whose existence is guaranteed by Lemma \ref{lem:finite}. Let $M$ be $\max_{j\in J}\{||v[d^j]_o||\}$ and $m$ be $\min_{j\in J}\{||v[d^j]_o||\}$. Let $\{I_i\}_{i\in I}$ be an open cover over $[0,1]$ such that, for any $i\in I$, $f(I_i)\subset W_{j_i}$ for some $j_i\in J$. Let $\{\xi_i\}$ be a partition of unity subordinate to $\{I_i\}$. Define $V_f=\sum_{i\in I}\xi_i V_{j_i}$ a descent vector field on $f([0,1])$. By linearity of the first variation of length, the first variation of length on the direction of $\sum_{i\in I}\xi_i V_{j_i}$ is still bounded from below by $-2M$, and from above by $-\frac{m}{2}$. Since $f([0,1])\subset g^{\le L}_N$, flowing along $V_f$ for time $s_V:=\frac{s}{2M}$ amounts to flowing along the normalized field for time at most $s$, hence does not exceed $\hat{\Gamma}^{\le L}_N$.
	
	 Thus the second step of the homotopy will be constructed as flowing along $V_f$ for time $s_V$.
	
	This second step decreases the maximal length by at least $\frac{m}{2}\,s_V=\frac{m\,s}{4M}$ after flow time $s_V$. We cannot directly re-iterate this step two since $V_j$'s in the construction are defined on $g^{\le L}_N$ while we only know that after the flow in step two, the orbifold 1-cycles are in $\hat{\Gamma}^{\le L}_N$. This will be solved by applying the Birkhoff homotopy again. Thus we can iterate descent process by $V_f$ and Birkhoff homotopy until $f([0,1])$ reaches $\hat{\Gamma}^{<\epsilon}_N$. The re-iteration cannot proceed since the construction of open cover $I$ requires $f([0,1])$ to be entirely contained by $\hat{\Gamma}^{\le L}_N\backslash\hat{\Gamma}^{< \epsilon}_N$. In this case we modify the construction of open cover $\{I_i\}_{i\in I}$ so that it consists of $I_i$ with $f(I_i)\subset W_{j_i}$ and an $I_0$ with $f([0,1])\cap \hat{\Gamma}^{<\epsilon}_N\subset f(I_0)\subset\hat{\Gamma}^{<2\epsilon}_N$ for all $i\in I$. $V_f$ is constructed as $\sum_{i\in I}\xi_iV_{j_i}$ where $V_{j_i}$ is set to be $0$ if $f(I_i)\subset \hat{\Gamma}^{<2\epsilon}_N$. Then the flow of $V_f$ shrinks orbifold 1-cycles outside of $f(I_0)$ by at least $\frac{m\,s}{4M}$ after flow time $s_V$, and shrinks $f(I_0)$ by some amount that does not matter. Re-iteration of this step shrinks $f([0,1])$ all the way below $\hat{\Gamma}^{<2\epsilon}_N$. This finishes the proof, since the $\epsilon$ we started with could have been set as $\frac{\epsilon}{2}$.

\end{proof}


\subsection{Homotopy to Constant Cycles}

Let $\tilde{\Gamma}$ be the classical 1-cycle space on $|\mathcal{O}|$. Let $|[c]|$ be the classical 1-cycle corresponding to orbifold 1-cycle $[c]\in\Gamma$, constructed by simply by looking at the image of $[c]$ on $|\mathcal{O}|$. Let $\tilde{\Gamma}^i$ be the projection image of $\Gamma^i$ by $|\cdot|$.

\begin{thm}\label{thm:deform2}
	Let $\mathcal{O}$ be a compact 2-orbifold homeomorphic to $S^2$. Let $L$ be a positive number less than the length of the shortest non-trivial orbifold closed geodesic on $\mathcal{O}$. Let $f:[0,1]\rightarrow\hat{\Gamma}_N^{\le L}$ be a continuous map. Then there exists a homotopy $H:[0,1]\times[0,1]\rightarrow \tilde{\Gamma}$ such that $H(0,\cdot)=|f(\cdot)|$ and $H(1,\cdot)\subset\tilde{\Gamma}^{0}$.
\end{thm}
\begin{proof}
	The first step of the homotopy will simply be the projection of the homotopy from Theorem \ref{thm:deform1} on $|\mathcal{O}|$. Denote by $h$ a 1-family of classical 1-cycle from $[0,1]\sqcup[1,2]$ to $|\mathcal{O}|$. As the time-one endpoint of the homotopy in step one, $h\subset\tilde{\Gamma}^{\le \epsilon}_N$. 
	
	If there is a point $p$ that is not in $h([0,1])([0,1]\sqcup[1,2])$ the image of union of all 1-cycles in $h$ on $|\mathcal{O}|$, then a deformation retract on $|\mathcal{O}|\backslash \{p\}$, which is homeomorphic to $\mathbb{D}^2$, retracts all of $h$ to some constant point, thus finishes the proof. If such point $p$ does not exist, we perform the following homotopy to create it:
	
	Consider $\cup_{t\in[0,1]}h(t)(0^+)\cup_{t\in[0,1]:h(t)\in\tilde{\Gamma}^1} h(t)(1^+)$, ``stripes of basepoints'' on $|\mathcal{O}|$. The set is the collection of basepoints $0^+$ and $1^+$ of all the 1-cycles in $h$, but discarding the basepoint at $1^+$ when $h(t)$ is not of type two loops, since in this case one basepoint is enough. ``The stripes'' is a subset of a compact 1-dimensional set on $|\mathcal{O}|$. Take a regular point $x$ in $|\mathcal{O}|$ such that $x$ is not in ``the stripes'' and $B_{3\epsilon}(x)$ does not intersect the singular set of $\mathcal{O}$, where $\epsilon$ is the $\epsilon$ from Theorem \ref{thm:deform1} as used in the first step. If $\epsilon$ was chosen small enough, neighborhood $B:=B_{3\epsilon}(x)$ is ``almost Euclidean''. 
	
	We will first consider the Euclidean case. Let $F$ be the deformation of $B_{3\epsilon}(x)$ on the radial direction that gradually pushes the two stripes out of $B_{\epsilon}(x)$ while fixing $B\backslash B_{2\epsilon}(x)$. $F$ extends to a homotopy on $|\mathcal{O}|$ by fixing everything outside the $B_{2\epsilon}(x)$. Let $G$ be the homotopy from $[0,1]$ to $\tilde{\Gamma}^{<\epsilon}$ that sends $h(t)(r)$ to $h(t)(r)+F\circ h(t)(0^+)-h(t)(0^+)$ for $r\in[0,1]$ and sends $h(t)(r)$ to $h(t)(r)+F\circ h(t)(1^+)-h(t)(1^+)$ for $r\in[1,2]$ if $h(t)\in\tilde{\Gamma}^1$, sends $h(t)(r)$ to $h(t)(r)+F\circ h(t)(0^+)-h(t)(0^+)$ for $r\in[0,1]\sqcup[1,2]$ otherwise. In simpler terms, during $G$, the basepoints carry their loops or cycles with them via linear translation, while the movement of basepoints are determined by $F$. Since the length of 1-cycles is less than $\epsilon$, after $G$, the family of 1-cycles can not reach $x$.
	
	As for the non-Euclidean case, we map $h(t)$ with $\exp_x^{-1}$ to get a family of 1-cycle in $B_{3\epsilon}(0)$ for those $h(t)$ with basepoint inside $B_{2\epsilon}(x)$, and do the Euclidean version of construction to get a homotopy for these 1-cycles, then map the homotopy back to the $B_{3\epsilon}(x)$ with $\exp_x$, then extend to a homotopy for the entire $h(t)$ by fixing the $h(t)$'s not affected. Since $B_{3\epsilon}(x)$ is a precompact set in the manifold part of $\mathcal{O}$, there exists a bound for the absolute value of sectional curvature everywhere, say $|sec|<R$. $\exp_x$ is $(1+R(3\epsilon)^2)-$bi-Lipschitz on $B_{3\epsilon}(0)$. Since $\epsilon$ can be arbitrarily small, the length distortion of the exponential map is small. Namely, the points in the final 1-cycle denoted by $h^1$ is at least $\epsilon-\frac{(1+R(3\epsilon)^2)\epsilon}{2}$ away from $x$, which is positive if $\epsilon$ is small.

\end{proof}	

We have the following immediate corollary.

\begin{cor}\label{thm:cor-loopunderlying}
	Let $|\mathcal{O}|$ be homeomorphic to $S^2$. Let $L<\ell(\mathcal{O})$. Any ordered piecewise-geodesic orbifold 1-cycle with 2 segments on the underlying space $|\mathcal{O}|$ with length no longer than $L$ can be homotoped to a constant 1-cycle. In particular, the same holds for ordered piecewise-geodesic loops.
\end{cor}

\section{Proof of Theorem}
\label{sec:proof}

In this section we prove the main theorem.
\begin{theorem*}
	For any compact Riemannian 2-orbifold homeomorphic to $S^2$, denoted by $\mathcal{O}$, $\ell(\mathcal{O})\le 4D(\mathcal{O})$, where $\ell(\mathcal{O})$ is the length of the shortest non-trivial closed orbifold geodesic, and $D(\mathcal{O})$ is the diameter.
\end{theorem*}
The idea of the proof is the following: Using topological property of Riemannian 2-orbifolds homeomorphic to $S^2$, we construct a family of non-contractible ordered piecewise-geodesic orbifold 1-cycles no longer than $4D(\mathcal{O})$. Assume the length of the shortest geodesic $\ell(\mathcal{O})$ is greater than $4D(\mathcal{O})$, by deformation results of the last section, we can homotope the family of 1-cycles to a fixed point on $|\mathcal{O}|$, which contradicts the non-contractibility of the family.


We first construct a specific 2-simplex on $|\mathcal{O}|$ that is not null-homotopic. 

Choose a fine triangulation on the underlying topological space $|\mathcal{O}|$ of $\mathcal{O}$ such that vertices are chosen from regular points, edges of triangles are minimizing geodesic between vertices, and the length of each edge of triangles is less than $\delta$ as chosen in Section \ref{subsec:Birkhoff}. The triangulation induces a simplicial complex structure on $|\mathcal{O}|$: $0$-simplices are vertices, denoted $x_i$ for $i\in I$ for some index set $I$. $1$-simplices are minimizing geodesic edges $x_ix_j$'s. According to the structure theorem by Borzellino \cite{borzellino1992riemannian}*{Proposition 31}, minimizing geodesics between regular points cannot travel through singular strata. Therefore the 1-simplex 
$x_ix_j$ stays within the principal stratum of $\mathcal{O}$. $x_ix_j$ corresponds to a unique orbifold curve $[x_ix_j]$, since it lies in the regular part of the orbifold. Also, since all the vertices are regular, we can concatenate $[x_ix_j]$ $[x_jx_k]$ $[x_kx_i]$ to get an orbifold free loop, denoted by $[c_{ijk}]$. Note that $[c_{ijk}]$ represents an ordered piecewise geodesic orbifold loop $[c_{ijk}]_o$ with $N$ breaks. Since the length of $[c_{ijk}]$ is less than $3\delta<\ell(\mathcal{O})$, by Corollary \ref{thm:cor-loopunderlying}, we have a homotopy $H_{ijk}$ from $|[c_{ijk}]_o|$ to a constant loop at a point, denoted by $p_{ijk}$. As in the proof of \cite{NR02}*{Theorem 1.1}, the homotopy contracting a loop to a point can be used to construct a 2-simplex $x_ix_jx_k:\Delta^2\rightarrow|\mathcal{O}|$ filling the interior of $x_ix_j+x_jx_k+x_kx_i=:\partial(x_ix_jx_k)$.

Next we try to extend the simplicial complex structure by adding a regular point $x_0$ as an extra $0$-simplex on $|\mathcal{O}|$. For any $x_i$ and $x_0$, a minimizing geodesic $x_0x_i$ from $x_0$ to $x_i$ is chosen as an extra 1-simplex on $\mathcal{O}$. For any 1-simplex $x_ix_j$, the length of $[c_{ij0}]_o$ is bounded from above by $2D(\mathcal{O})+\delta<4D(\mathcal{O})<\ell(\mathcal{O})$ thus satisfies the assumption of Corollary \ref{thm:cor-loopunderlying}. Then the extra 2-simplices $x_ix_jx_0$'s are constructed by filling the interior of $\partial(x_0x_ix_j)$ with homotopies $H_{ij0}$'s that shrink $[c_{ij0}]_o$ to a constant loop. As for 3-simplices, we have the following lemma.

\begin{lem}
	There exists a 2-simplex $x_ix_jx_k$ such that there exists no 3-simplex $x_ix_jx_kx_0$ with faces $x_ix_jx_k$  $x_jx_kx_0$ $x_ix_kx_0$ and $x_ix_jx_0$.
\end{lem}
\begin{proof}
	Suppose that for any 2-simplex $x_ix_jx_k$, there exists a 3-simplex $x_ix_jx_kx_0$, gluing together all such 3-simplices gives a simplicial complex $\mathbb{D}^3\rightarrow |\mathcal{O}|$, where each piece $x_ix_jx_kx_0$ is identified with a map from a radial 3-simplex in $\mathbb{D}^3$ to $|\mathcal{O}|$ as shown in Figure \ref{fig:extension}.
	
	\begin{figure}[!ht]
		\centering
		\includegraphics[scale=0.4]{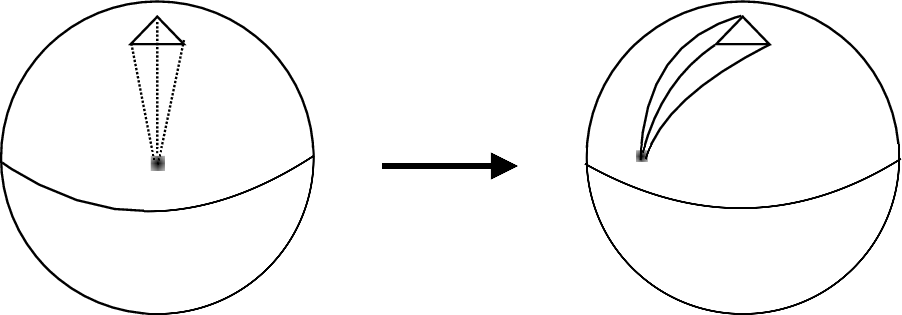}
		\caption{Extension of Skeleton}\label{fig:extension}
		
	\end{figure}
	
	The simplicial complex has boundary $S^2\rightarrow|\mathcal{O}|$, which is also $\sum x_ix_jx_k$, which can be viewed as the identity map of $S^2$. However $\mathbb{D}^3$ is contractible. There is a contradiction.
\end{proof}

For simplicity, say $x_1x_2x_3$ is the 2-simplex that does not extend. Also, denote $x_{ijk}$ as the point on $|\mathcal{O}|$ where the constant loop $H_{ijk}(1)$ lies. 


\begin{figure}[!ht]
	\centering
	\includegraphics[scale=0.6]{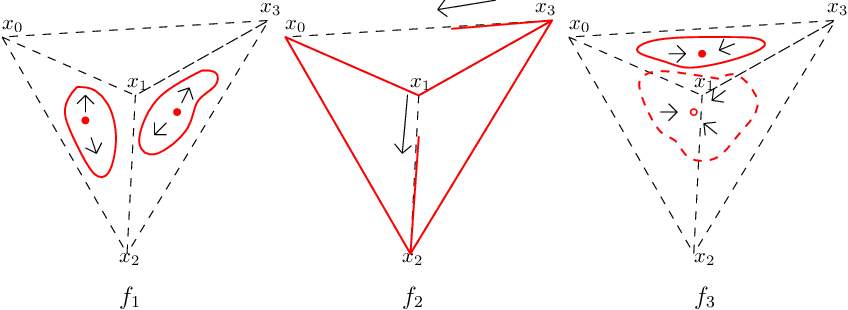}
	\caption{Family $f$}\label{fig:family}
\end{figure}

Now we construct the family $f:[0,1]\rightarrow\hat{\Gamma}^{\le 4D(\mathcal{O})+2\delta}_{3N}$ from three components. The construction follows the idea as in the proof of \cite{NR02}*{Theorem 1.1}.

(i) $f_1:[0,1]\rightarrow \hat{\Gamma}^{\le 4D(\mathcal{O})+2\delta}_{3N}$ is defined as the map sending $t$ to $H_{102}\sqcup H_{123}(1-t)$.

(ii) Let $[c_{102123}]_o$ be the orbifold 1-cycle of type II constructed by the concatenation of $[x_1x_0]$ $[x_0x_2]$ $[x_2x_1]$ $[x_1x_2]$ $[x_2x_3]$ and $[x_3x_1]$. Note that there is a backtrack along $[x_1x_2]$ in the middle. $f_2:[1,\frac{4}{3}]\rightarrow \hat{\Gamma}^{\le 4D(\mathcal{O})+2\delta}_{3N}$ is defined as the map which reduces the backtrack $[x_2x_1]*[x_1x_2]$ linearly until it is completely cancelled at $t=\frac{4}{3}$, which in the end gives us $[c_{1023}]_o$. Note that $[c_{1023}]_o$ is geometrically identical to $[c_{0231}]_o$. We define $f_2$ on $[\frac{4}{3},\frac{5}{3}]$ as a reparametrization from $[c_{1023}]_o$ to $[c_{0231}]_o$. Such reparametrization can be defined in a linear way without adding extra breakpoints. $f_2$ on $[\frac{5}{3},2]$ is defined as a linear homotopy which deforms $[c_{0231}]_o$ by extending out a backtrack along $x_0x_3$ linearly until we reach $[c_{023031}]_o$.

(iii)  We define $f_3:[2,3]\rightarrow\hat{\Gamma}^{\le 4D+2\delta}_{3N}$ as the map sending $t$ to $H_{023}\sqcup H_{031}(t-2)$ for $t\in[2,3]$.

Note that we cannot directly glue $f_i$'s together to get $f$, since $f_1(1)$ and $f_2(2)$ do not match with $f_2(1)$ and $f_3(2)$, respectively. In fact $f_1$ and $f_3$ are families of type I orbifold 1-cycle while $f_2$ consists of type II's. To fix this we put instead $(p_2)^{-1}[c_{102123}]_o$ at $t=1$ and $(p_2)^{-1}[c_{023031}]_o$ at $t=2$ for $f$. Here $p_1:\hat{\Gamma}_3\rightarrow\hat{\Gamma_1}$ and $p_2:\hat{\Gamma}_3\rightarrow\hat{\Gamma_2}$ are induced straightforwardly by $p_1:\Gamma_3\rightarrow\Gamma_1$ and $p_2:\Gamma_3\rightarrow\Gamma_2$, respectively. One easily checks that $(p_2)^{-1}$ is non-empty for a $[c]_o$ if $|[c]|(1)=|[c]|(2)$, and unique if $|[c]|(1)$ is regular. On the rest of the domain, we set $f=f_1$ on $[0,1)$, $f=f_2$ on $(1,2)$, and $f=f_3$ on $(2,3]$.

The continuity of $f$ within each stage follows from the continuity of homotopy $H_{ijk}$'s, the continuity of reparametrization, the continuity of backtrack cancellation and extension. The continuity of $f$ at $t=1$ and $t=2$ is a consequence of $p_1\circ f(1)=f_1(1)$, $p_2\circ f(1)=f_2(1), p_2\circ f(2)=f_2(2)$, and $p_1\circ f(2)=f_3(2)$.

Now we are ready to prove the main theorem.

\begin{proof}[Proof of the Main Theorem]
	Under the assumption that $\ell(\mathcal{O})>4D(\mathcal{O})+2\delta$, the $f$ constructed above, as shown in figure \ref{fig:family}, sweeps out the entire  $\partial(x_0x_1x_2x_3):\Delta^2\rightarrow |\mathcal{O}|$. Following \cite{almgren1962homotopy}, $f$ corresponds to a 2-simplex homologous to $\partial(x_0x_1x_2x_3)$. By Almgren's isomorphism Theorem, the non-contractibility of $\partial(x_0x_1x_2x_3)$ implies the non-contractibility of $f$. On the other hand, by Theorem \ref{thm:deform2} there exists a homotopy $H:[0,1]\times[0,3]\rightarrow \hat{\Gamma}$ such that $H(0,\cdot)=|f|(\cdot)$ and $H(1,\cdot)$ is constant at some $q\in|\mathcal{O}|$, which contradicts the non-contractibility. Therefore $\ell(\mathcal{O})\le 4D(\mathcal{O})+2\delta$. Note that $\delta>0$ can be chosen to be arbitrarily small. Therefore $\ell(\mathcal{O})\le 4D(\mathcal{O})$.
\end{proof}

Using the argument in the theorem, we also have the following corollary.
\begin{cor}\label{cor:last}
	Let $\mathcal{O}$ be a compact Riemannian 2-orbifold with a finite orbifold fundamental group, $\ell(\mathcal{O})\le 8D(\mathcal{O})$.
\end{cor}
\begin{proof} 
	The corollary follows from the classification of 2-orbifolds \cite{thurston2022geometry}: a 2-orbifold $\mathcal{O}$ with a finite orbifold fundamental group either is homeomorphic to $S^2$, or admits a double cover $\pi$ by a 2-orbifold $\hat{\mathcal{O}}$ homeomorphic to $S^2$. By \cite{La20} the double cover can be made into a Riemannian double cover. $\ell(\hat{\mathcal{O}})$ is no greater than $4D(\hat{\mathcal{O}})$. Denote this closed orbifold geodesic by $[\hat{c}]$. $[\hat{c}]$ projects down to a closed orbifold geodesic $[c]$ on $\mathcal{O}$ by $\pi$. Therefore $\ell(\mathcal{O})\le \Length[c]\le\Length[\hat{c}]= l(\hat{\mathcal{O}})$. 
	
	On the other hand, $D(\hat{\mathcal{O}})\le2D(\mathcal{O})$. Suppose not, let $r$ be a number with $D(\hat{\mathcal{O}})>2r>2D(\mathcal{O})$. Say $p,q$ on $\hat{\mathcal{O}}$ have distance greater than $2r$, then $B_r(p)$ and $B_r(q)$ are disjoint balls. Both $B_r(p)$ and $B_r(q)$ project down to all of $\mathcal{O}$ since $r>D(\mathcal{O})$. Since any point in $\mathcal{O}$ has at most two preimages by $\pi$, we have $B_r(p)\sqcup B_r(q)=\hat{\mathcal{O}}$, contradicting the connectedness.
	
\end{proof}

\appendix

\section{Continuity of the Birkhoff Homotopy}

The following lemma is useful in modifying representatives of orbifold 1-cycles.

\begin{lem}\label{lem:perturb}
	Let $[c]$ be an $L$-Lipschitz orbifold 1-cycle with representative $c=(c_i,g_i)$ over the subdivision $0=t_0\le t_1\le...\le t_N=1\le...\le t_{2N}=2$. Let $0=t^1_0\le...\le t^1_N=1\le...\le t^1_{2N}=2$ be another subdivision with $|t_i-t^1_i|<\epsilon$ for any $i$ for some small enough $\epsilon>0$. Then there is a new representative $c^{1}$ of $[c]$ over the new subdivision such that $c^1_i$ and $c_i$ are in the same component of $X$ and $g_i$ and $g^1_i$ are germs of the same local isometry.
\end{lem}
\begin{proof}
	First assume $t_1<t^1_1$. For any $i$, $g_i$ extends to a local isometry $\bar{g}_i$ on a small radius $\epsilon_i$. Let $\epsilon>0$ be smaller than $\dfrac{\min_i\{\epsilon_i\}}{L}$. Then the first segment $c^{1}_1$ can be constructed by gluing $c_1$ with $\bar{g}_1\circ c_2|_{[t_1,t^1_1]}$. The perturbed $g^1_1$ is the germ of $\bar{g}_1$ at $c_2(t^1_1)$. If $t_1>t^1_1$, $c^1_1$ will simply be $c_1|_{[0,t^1_1]}$. In other words, local isometries $\bar{g}_i$ allow us to cut off and transfer segments where the two subdivisions do not align. Do the same for all other segments and groupoid elements to get the new representative $c^{1}$.
\end{proof}

A direct consequence is the following lemma, which shows that for piecewise-geodesic orbifold 1-cycles, it suffices to check convergence on subdivision points between geodesics and groupoid elements.

\begin{lem}\label{lem:ez}
	Let $[c^0]$ and $[c^j]$'s be orbifold 1-cycles of the same type with representatives $c^0=(g^0_i,c^0_i)$ and $c^j=(g^j_i,c^j_i)$ such that $c^0_i$'s and $c^j_i$'s are unique minimizing geodesics for all $i$ and $j$, and $t^j_i$'s and $t^0_i$'s are the subdivision points. Then $g^j_i\rightarrow g^0_i$ and $t^j_i\rightarrow t^0_i$ for all $i$ imply $[c^j]\rightarrow [c^0]$.
\end{lem}

\begin{proof}
	For large enough $j$, the subdivision of $c^j$ is close to that of $c^0$. We construct new representatives $c^{j,1}$ as in Lemma \ref{lem:perturb}. By the Lipschitz property and $g^j_i\rightarrow g^0_i$, we have $g^{j,1}_i\rightarrow g^0_i$. Again by the Lipschitz property, we have that $c^{j,1}$ is in the same modeling neighborhood of $T_{c^0}\Gamma$ and pointwise convergence $c^{j,1}_i\rightarrow c^0_i$ for any $i$.
\end{proof}

From now on, let $s^j\rightarrow s^0$ in $[0,1]$ and $[c^j]\rightarrow[c^0]$ in $\Gamma^{\le L}_N$, in other words, $(s^j,[c^j])\rightarrow(s^0,[c^0])$ in $[0,1]\times\Gamma^{\le L}_N$. Then for $j$ large enough, there exists  a modeling representative $c^j=(c^j_i,g^j_i)$ of $[c^j]$ with respect to the representative $c^0=(c^0_i,g^0_i)$ of $[c^0]$, where $c^j$ and $c^0$ are defined over the same subdivision $0=t_0<t_1<...<t_{2N}=2$, such that, $c^j_i\rightarrow c^0_i$ pointwise and $g^j_i\rightarrow g^0_i$ for $i=1,2,...,2N$. 

Recall that $\Psi^1[c]\circ P_{[c]}=[c]$. For simplicity, denote $P_{[c^0]}$ by $P_0$, and $P_{[c^j]}$ by $P_j$. Recall that $f_s^{[c]}=s\cdot \text{id}_{[0,2]}+(1-s)P_{[c]}$. For simplicity, denote $f_{s^0}^{[c^0]}$ by $f_0$, and $f_{s^j}^{[c^j]}$ by $f_j$. It is an elementary exercise that $P_j\rightarrow P_0$ pointwise and $f_j\rightarrow f_0$ pointwise. Since $P_j$'s are not necessarily one-to-one, they might not have inverses. However they are piecewise-linear maps, so each admits a unique upper semi-continuous inverse, which is monotonically increasing piecewise-linear and has finitely many jump discontinuities. We abuse notation and call such inverses $P_j^{-1}$'s. On the other hand $f_j$'s have inverses except when $s^j=0$. One can check that $(f^{-1}\circ P)_j$ defined as $f^{-1}_j\circ P_j$ if $s^j\ne0$ and to be the identity map if $s^j=0$, is a continuous Lipschitz function. One can check that $(f^{-1}\circ P)_j\rightarrow (f^{-1}\circ P)_0$ pointwise.

\begin{lem}
	$\Phi^1(s^j,[c^j])\rightarrow \Phi^1(s^0,[c^0])$ pointwise.
\end{lem}
\begin{proof}
	$c^j\circ P^{-1}_j$'s are representatives for $\Psi^1[c^j]$'s and $c^0\circ P^{-1}_0$ is a representative for $\Psi^1[c^0]$. Also, $\Phi^1(s^j,[c^j])$ admits a representative $d^j:=c^j\circ P_j^{-1}\circ f_j$ with subdivision points $(f^{-1}\circ P)_j(t_i)$ and groupoid elements $g^j_i$'s for all $j$ (and for $j=0$). We have $(f^{-1}\circ P)_j(t_i)\rightarrow (f^{-1}\circ P)_0(t_i)$ for all $i$ by the previous discussion. We also have that $g^j_i\rightarrow g^0_i$ by the assumption of $c^j\rightarrow c^0$. We cannot use Lemma \ref{lem:ez} to conclude convergence just yet, since $d^j$ is not necessarily piecewise geodesic with the same number of breaks. Geometrically, this is because $\Phi^1$ gradually enlarges the ``densely parametrized'' (faster) segments of $[c]$ while shrinking the ``loosely parametrized'' (slower) ones (in order to get a constant-speed orbifold 1-cycle at $s=1$ eventually), thus introducing new non-geodesic points if $s\ne0,1$. These non-geodesic points mark the boundaries of domains of different speed. However these new non-geodesic points for $\Phi^1(s,[c])$ will always be at the original breakpoints $t_i$'s. Let $d^{j,1}$ be the refinement of $d^j$ by adding subdivision points $t_i$'s and identity groupoid elements $\text{id}_{d^j(t_i)}$'s at $t_i$'s. Notice that $d^j(t_i)=\Psi^1(c^j)\circ f_j(t_i)$, then by the Lipschitz property of $\Psi^1(c^j)$ and $f_j$ and pointwise convergence $\Psi^1(c^j)\rightarrow \Psi^1(c^0)$ and $f_j\rightarrow f_0$, we have that $d^j(t_i)\rightarrow d^0(t_i)$ and $\text{id}_{d^j(t_i)}\rightarrow\text{id}_{d^0(t_i)}$.
	Therefore by Lemma \ref{lem:ez}, $\Phi^1(s^j,[c^j])\rightarrow \Phi^1(s^0,[c^0])$ pointwise.
\end{proof}

The continuity of $\Phi^2$ on $\Phi^1(\Gamma^{\le L}_N)$ follows from the same argument: we verify convergence on the breakpoints corresponding to those of the original $[c^j]$'s, then verify convergence on the newly introduced non-geodesic breakpoints, and finally Lemma \ref{lem:ez} confirms pointwise convergence.


\begin{ack}
This work would not have been possible without the professional guidance and personal support of my supervisor, Marco Radeschi. I am heavily indebted to Marco for his dedication to share his expertise, as well as his unwavering encouragement.
\end{ack}



\end{document}